\documentclass[a4paper,11pt]{amsart}

\usepackage{mathrsfs}
\usepackage{cite}
\usepackage{amssymb}
\usepackage{graphicx}

\DeclareMathAlphabet{\mathpzc}{OT1}{pzc}{m}{it}

\newtheorem{theorem}{Theorem}[section]
\newtheorem{lemma}[theorem]{Lemma}

\newtheorem{corollary}[theorem]{Corollary}

\newtheorem{proposition}[theorem]{Proposition}
\theoremstyle{definition}
\newtheorem{definition}[theorem]{Definition}

\theoremstyle{remark}
\newtheorem{remark}[theorem]{Remark}

\newtheorem{claim}[theorem]{Claim}

\numberwithin{equation}{section}

\allowdisplaybreaks[4]

\begin{document}

\title{Geometry of non-degenerate locally tame non-isolated singularities} 

\author{Christophe Eyral and Mutsuo Oka}
\address{C. Eyral, Institute of Mathematics, Polish Academy of Sciences, \'Sniadeckich~8, 00-656 Warsaw, Poland}
\email{ch.eyral@impan.pl}
\address{M. Oka, Department of Mathematics, Tokyo University of Science, 1-3 Kagurazaka, Shinjuku-ku, Tokyo 162-8601, Japan}   
\email{oka@rs.tus.ac.jp}



\begin{abstract}
We give a criterion to test geometric properties such as Whitney equisingularity and Thom's $a_f$ condition for new families of (possibly non-isolated) hypersurface singularities that ``behave well'' with respect to their Newton diagrams. As an important corollary, we obtain that in such families all members have isomorphic Milnor fibrations. 
\end{abstract}

\maketitle

\markboth{C. Eyral and M. Oka}{Geometry of non-degenerate locally tame non-isolated singularities} 

\section{Introduction}\label{intro}

Let $f(t,\mathbf{z})=f(t,z_1,\ldots, z_n)$ be a non-constant polynomial function on $\mathbb{C}\times \mathbb{C}^n$ such that $f(t,\mathbf{0})=0$ for all small $t$. As usual, we write $f_t(\mathbf{z}):=f(t,\mathbf{z})$ and we denote by $V(f_t)$ the zero set of $f_t$, which defines a hypersurface in $\mathbb{C}^n$. The goal of this paper is to investigate the local geometry, at singular points, of the hypersurfaces $V(f_t)$ as the parameter $t$ varies. Here, the word ``local'' always refers to properties associated to (small representatives of) germs at the origin. Our main result is a criterion to test Whitney equisingularity and Thom's $a_f$ condition for new families of (possibly non-isolated) hypersurface singularities that ``behave well'' with respect to their Newton diagrams.

To our knowledge, the first result in this direction concerns \emph{isolated} singularities and was obtained by Brian\c con in an unpublished lecture notes \cite{B}. It asserts that, for isolated singularities, the family $\{V(f_t)\}_t$ is Whitney equisingular if the Newton boundary of $f_t$ is independent of $t$ and $f_t$ is non-degenerate (in the sense of Kouchnirenko \cite{K}) for all small $t$. Here, the expression ``Whitney equisingular'' means that there exists a Whitney $(b)$-regular stratification (in the sense of \cite{GWPL})\footnote{In particular, we do not require the frontier condition to be satisfied. However, note that taking the connected components of a Whitney $(b)$-regular stratification gives a new such a stratification for which the frontier condition holds (see \cite{GWPL}).} of the  hypersurface $V(f):=f^{-1}(0)$ (which lies in $\mathbb{C}\times \mathbb{C}^n$) such that the $t$-axis $\mathbb{C}\times \{\mathbf{0}\}$ is a stratum. For families of isolated singularities, this simply means that $V(f)\setminus (\mathbb{C}\times \{\mathbf{0}\})$ is smooth and Whitney $(b)$-regular over $\mathbb{C}\times \{\mathbf{0}\}$. Note that Whitney equisingularity is quite a strong form of equisingularity. In particular, combined with the Thom--Mather first isotopy theorem \cite{T,Ma}, it implies topological equisingularity, that is, the local, ambient, topological type of $V(f_t)$ at $\mathbf{0}$ is independent of $t$ for all small $t$.

For non-isolated singularities, Whitney equisingularity is more delicate. In particular, for such singularities, the smooth part of $V(f)$ may be Whitney $(b)$-regular over the $t$-axis without the family $\{V(f_t)\}_t$ being Whitney equisingular or even without being topologically equisingular. For instance, look at Zariski's example $f(t,z_1,z_2)=t^2z_1^2-z_2^2$ for which the smooth part is actually Whitney $(b)$-regular over $\mathbb{C}\times \{\mathbf{0}\}$, while the local, ambient, topological type of $V(f_t)$ at $\mathbf{0}$, for $t\not=0$, is different from that of $V(f_0)$ (see \cite{Z}).

In \cite[Theorem 3.8]{EO}, we gave a first generalization of Brian\c con's theorem for families of \emph{non-isolated} singularities. Unlike the isolated case, for non-isolated singularities, we should take into account not only the compact faces of the Newton polyhedron (via non-degeneracy) but also the non-compact faces. This led us to the notion of ``uniform local tameness'' which is a kind of non-degeneracy condition, uniform in $t$, on the non-compact faces of the Newton polyhedron (see below). Roughly, our theorem says that if $\{f_t\}_t$ is a family of non-degenerate functions with constant Newton boundary and satisfying the uniform local tameness condition, then the corresponding family of hypersurfaces $\{V(f_t)\}_t$ is Whitney equisingular.

Unfortunately this theorem has some restriction. Indeed, when $f_t(\mathbf{z})$ is the product of two functions  $g_t(\mathbf{z})$ and  $h_t(\mathbf{z})$ such that the dimension at $\mathbf{0}$ of the intersection $V(g_t)\cap V(h_t)$\textemdash which is contained in the singular locus of $V(f_t)$\textemdash is greater than or equal to $1$, then $f_t$ is never non-degenerate if its Newton diagram intersects each coordinate axes (see \cite{K}), and hence Theorem 3.8 of \cite{EO} does not apply to this case. It was therefore very desirable to find a new criterion that includes the (much larger) class of functions of the form $f_t(\mathbf{z})=f_t^1(\mathbf{z})\cdots f_t^{k_0}(\mathbf{z})$. That is the main, but not unique, goal of the present paper. 

More precisely, we shall introduce the class of ``Newton-admissible'' families, which are defined as follows. A family
 $\{f_t\}_t$ is Newton-admissible if $f_t(\mathbf{z})$ is the product of non-degenerate functions $f_t^1(\mathbf{z}),\ldots,f_t^{k_0}(\mathbf{z})$ with constant Newton boundary (with respect to $t$) and such that any subfamily 
\begin{equation*}
\{V(f_t^{k_1},\ldots,f_t^{k_p})\}_t:=\bigg\{ \bigcap_{j=1}^p V(f_t^{k_j})\bigg\}_t, 
\end{equation*}
with $\{k_1,\ldots,k_p\}\subseteq\{1,\ldots,k_0\}$ is a family of non-degenerate complete intersection varieties satisfiying the ``uniform local tameness'' condition. 
Essentially, the latter condition means the following. First, for a fixed $t$, the local tameness condition guarantees  that for any subset $I\subseteq \{1,\ldots,n\}$ with
\begin{equation*}
f^{k_1}\vert_{\mathbb{C}^I}\equiv\cdots\equiv f^{k_p}\vert_{\mathbb{C}^I}\equiv 0,
\end{equation*}
any weight vector $\mathbf{w}=(w_1,\ldots,w_n)\in \mathbb{N}\setminus \{\mathbf{0}\}$ whose zero weights $w_i=0$ are located at the places $i\in I$, and any non-zero small enough complex numbers $u_i$, $i\in I$ (say, $\sum_{i\in I}\vert u_i\vert^2< R_t$ for some $R_t>0$), the $k_p$-form 
\begin{equation*}
df^{k_1}_{t,\mathbf{w}}\wedge \cdots\wedge df^{k_p}_{t,\mathbf{w}}
\end{equation*}
is nowhere vanishing  in the toric variety 
\begin{equation*}
V(f^{k_1}_{t,\mathbf{w}},\ldots,f^{k_p}_{t,\mathbf{w}})\cap \{\mathbf{z}\in\mathbb{C}^{*n}\mid (z_i)_{i\in I}=(u_i)_{i\in I}\}.
\end{equation*}
Here, $f^{k}_{t,\mathbf{w}}$ denotes the face function of $f^{k}_{t}$ with respect to the weight vector $\mathbf{w}$. A number $R_t$ for which the above property holds is called a radius of local tameness of the set of functions $\{f^{k_1}_t,\ldots,f^{k_p}_t\}$. Now, when the parameter $t$ varies, we say that the family of varieties $V(f_t^{k_1},\ldots,f_t^{k_p})$ satisfies the ``uniform'' local tameness condition if there are radii of local tameness $R_t$ that can be chosen to be independent of $t$ for all small $t$ (for a precise definition, see Sections \ref{ltndcip} and \ref{sect-we}).
With this terminology, our main result says that if $\{f_t\}_t$ is Newton-admissible, then the family of hypersurfaces $\{V(f_t)\}_t$ is Whitney equisingular; moreover, the corresponding Whitney stratification of $V(f)$ is just the canonical toric stratification (see Theorem \ref{mt}). In the case of a single function (i.e., $k_0=1$), the result reduces to Theorem 3.8 of \cite{EO}.

While the proof of this equisingularity criterion is not easy, its statement is quite simple. Let us emphasize that the assumptions (non-degeneracy and uniform local tameness) are elementary algebraic conditions which can often be checked using a computer. On the other hand, the conclusion (which is geometric in nature) cannot be directly checked by a computer. 

Besides Whitney equisingularity, another important geometric property is so-called Thom's $a_f$ condition. For instance, this condition plays a crucial role in the Thom--Mather second isotopy theorem (see \cite{Ma,T}). It also ensures the transversality of the nearby fibres of (germs at $\mathbf{0}$ of) analytic functions to small spheres, which is a key property to prove the existence of their Milnor fibrations. By a theorem of Parusi\'nski \cite{P} and Brian\c con--Maisonobe--Merle \cite{BMM}, Whitney's $(b)$-regularity implies so-called $w_f$ condition (also called Thom's \emph{strict} condition), which in turn implies Thom's $a_f$ condition. Thus, combined with our equisingularity criterion (Theorem \ref{mt}), it follows that if $\{f_t\}_t$ is a Newton-admissible family, then the canonical toric stratification of $V(f)$ satisfies Thom's $a_f$ condition (see Theorem \ref{mt2}). This theorem can also be proved by a direct explicit calculation using only elementary methods (see Section \ref{sect-proofmt2}). 

Finally, as an important corollary of Theorems \ref{mt} and \ref{mt2}, we shall show that if the family $\{f_t\}_t$ is Newton-admissible, then for $t$ small enough the Milnor fibrations of $f_t$ and $f_0$ at $\mathbf{0}$ are isomorphic (see Theorem \ref{mt4}). A key observation in the proof of this result is Theorem \ref{mt3} which asserts that Newton-admissible families have so-called ``uniform stable radius''. Roughly, a family has a uniform stable radius if there exists a radius $r>0$ such that the nearby fibres $V(f_t-\eta):=f_t^{-1}(\eta)$, $\eta\not=0$, are non-singular in the open ball $\mathring{B}_r$, transversely intersect all spheres $\mathbb{S}_{\varepsilon''}$ with radius $\varepsilon''\in [\varepsilon',\varepsilon]\subseteq (0,r]$, and if this occurs ``uniformly'' with respect to the deformation parameter $t$ (for a precise definition, see Section \ref{sect-MF}). In turn, the proof of Theorem \ref{mt3} is based on a ``uniform'' version of a well-known result according to which any sufficiently small sphere transversely intersects all strata of any given Whitney stratification (see Proposition \ref{lemma-fpofmt3}).

To complete this introduction, note that working with a product of functions $f_t^1(\mathbf{z}),\ldots,f_t^{k_0}(\mathbf{z})$ presents an essential new difficulty. Indeed, to prove our main result on Whitney equisingularity, we need to calculate limits of expressions of the form
\begin{equation*}
df^{k_1}_{t,\mathbf{w}}(\rho(s))\wedge \cdots\wedge df^{k_p}_{t,\mathbf{w}}(\rho(s))
\end{equation*} 
along real analytic curves $\rho(s)$ as $s\to 0$, and in general, the limits of the linearly independent $1$-forms 
\begin{equation*}
df^{k_1}_{t,\mathbf{w}}(\rho(s)), \ldots, df^{k_p}_{t,\mathbf{w}}(\rho(s))
\end{equation*}  
may be linearly dependent. This phenomenon, which does not happen in the case of a single function, is one of the main difficulties.

\section{Non-degenerate locally tame complete intersection variety}\label{ltndcip}

Let $\mathbf{z}:=(z_1,\ldots, z_n)$ be coordinates for $\mathbb{C}^n$, and let $f(\mathbf{z})=\sum_\alpha c_\alpha\, \mathbf{z}^\alpha$ be a non-constant polynomial function which vanishes at the origin. Here, $\alpha:=(\alpha_1,\ldots,\alpha_n)$ is an integer vector, $c_\alpha\in\mathbb{C}$, and $\mathbf{z}^{\alpha}$ is a notation for the monomial $z_1^{\alpha_1}\cdots z_n^{\alpha_n}$. 
For any $I\subseteq\{1,\ldots, n\}$, we denote by $\mathbb{C}^I$ (respectively, $\mathbb{C}^{*I}$) the set of points $(z_1,\ldots, z_n)\in \mathbb{C}^n$ such that $z_i=0$ if $i\notin I$ (respectively, $z_i=0$ if and only if $i\notin I$).
In particular, we have $\mathbb{C}^{\emptyset}=\mathbb{C}^{*\emptyset}=\{\mathbf{0}\}$ and $\mathbb{C}^{*\{1,\ldots,n\}}=\mathbb{C}^{*n}$, where $\mathbb{C}^*:=\mathbb{C}\setminus \{\mathbf{0}\}$. Throughout this paper, we are only interested in a \emph{local} situation, that is, in (arbitrarily small representatives of) germs at the origin.

In this section, we recall the definition of non-degenerate, locally tame, complete intersection varieties already used in \cite{O3}. We start with the special case of hypersurfaces.

\subsection{Non-degenerate locally tame hypersurface}\label{sect-11}

The Newton polyhedron $\Gamma_{\! +}(f)$ of the germ of $f$ at the origin $\mathbf{0}\in \mathbb{C}^n$ (with respect to the coordinates $\mathbf{z}=(z_1,\ldots, z_n)$) is the convex hull in $\mathbb{R}_+^n$ of the set
\begin{equation*}
\bigcup_{c_\alpha\not=0} (\alpha+\mathbb{R}_+^n).
\end{equation*}
The Newton boundary (also called Newton diagram) of $f$ is the union of the compact faces of  $\Gamma_{\! +}(f)$. It is denoted by $\Gamma(f)$.

For any non-zero weight vector $\mathbf{w}:=(w_1,\ldots,w_n)\in\mathbb{N}^n\setminus\{\mathbf{0}\}$, let us denote by $\ell_\mathbf{w}$ the restriction to $\Gamma_{+}(f)$ of the linear map
\begin{equation*}
\mathbf{x}:=(x_1,\ldots, x_n)\in \mathbb{R}^n\mapsto\sum_{i=1}^n w_i x_i \in \mathbb{R}.
\end{equation*}
Let $d(\mathbf{w};f)$ be the minimal value of $\ell_\mathbf{w}$, and let $\Delta(\mathbf{w};f)$ be the face of $\Gamma_{+}(f)$ defined by 
\begin{equation*}
\Delta(\mathbf{w};f)=\{\mathbf{x}\in \Gamma_{+}(f)\mid \ell_\mathbf{w}(\mathbf{x})=d(\mathbf{w};f)\}.
\end{equation*} 
Note that if the $w_i$'s are positive for all $i$, then $\Delta(\mathbf{w};f)$ is a (compact) face of $\Gamma(f)$. Finally, put $I(\mathbf{w}):=\{i \in \{1,\ldots,n\}\mid w_i=0\}$. By definition, the \emph{non-compact Newton boundary} is the union of the usual Newton boundary $\Gamma(f)$ together with the ``essential'' non-compact faces of $\Gamma_+(f)$, that is, the non-compact faces $\Delta(\mathbf{w};f)$ such that the restriction of $f$ to $\mathbb{C}^{I(\mathbf{w})}$ identically vanishes.

The germ at $\mathbf{0}$ of the hypersurface $V(f):=f^{-1}(0)$ is called \emph{(Newton) non-degenerate} if for any positive weight vector $\mathbf{w}$, the toric hypersurface 
\begin{equation*}
V^*(f_{\mathbf{w}}):=\{\mathbf{z}\in\mathbb{C}^{*n}\mid 
f_{\mathbf{w}}(\mathbf{z})=0\}
\end{equation*}
is a reduced non-singular hypersurface in the complex torus $\mathbb{C}^{*n}$. This means that $f_{\mathbf{w}}$ has no critical point in $V^*(f_\mathbf{w})$\textemdash equivalently,  by Euler identity, in $\mathbb{C}^{*n}$\textemdash that is, the $1$-form $df_{\mathbf{w}}$ is nowhere vanishing in $V^*(f_\mathbf{w})$.
Here, we write $f_\mathbf{w}$ for the face function 
\begin{equation*}
\mathbf{z} \mapsto \sum_{\alpha\in \Delta(\mathbf{w};f)} c_\alpha\, \mathbf{z}^\alpha
\end{equation*} 
of $f$ with respect to $\mathbf{w}$.
We emphasize that $V^*(f_{\mathbf{w}})$ is globally defined in~$\mathbb{C}^{*n}$.

Let $\mathcal{V}_f$ be the set of all subsets $I\subseteq \{1,\ldots,n\}$ such that the restriction of $f$ to $\mathbb{C}^I$ identically vanishes.
We say that $\mathbb{C}^I$ is a vanishing (respectively, a non-vanishing) coordinate subspace for $f$ if $I\in\mathcal{V}_f$ (respectively, if $I\notin\mathcal{V}_f$).
For any $u_{i_1},\ldots, u_{i_m}\in\mathbb{C}^*$ ($m\leq n$), let $\mathbb{C}^{*n}(u_{i_1},\ldots, u_{i_m})$ denote the set of points $(z_1,\ldots,z_n)\in\mathbb{C}^{*n}$ such that $z_{i_j}=u_{i_j} \mbox{ for } 1\leq j\leq m$. 
We say that the germ of $V(f)$ at $\mathbf{0}$ is \emph{locally tame} if there is a positive number $R(f)$ such that for any non-empty subset $I:=\{i_1,\ldots, i_m\}\in \mathcal{V}_f$, any non-zero weight vector $\mathbf{w}$ with $I(\mathbf{w})=I$, and any non-zero complex numbers $u_{i_1},\ldots, u_{i_m}$ satisfying the inequality
\begin{equation*}
\sum_{j=1}^m\vert u_{i_j}\vert^2 < R(f), 
\end{equation*}
the toric hypersurface 
\begin{equation*}
V^*(f_{\mathbf{w}})\cap \mathbb{C}^{*n}(u_{i_1},\ldots, u_{i_m})
\end{equation*}
is a reduced non-singular hypersurface in $\mathbb{C}^{*n}(u_{i_1},\ldots, u_{i_m})$. This means that $f_{\mathbf{w}}$ has no critical point in $V^*(f_{\mathbf{w}})\cap \mathbb{C}^{*n}(u_{i_1},\ldots, u_{i_m})$\textemdash equivalently, in $\mathbb{C}^{*n}(u_{i_1},\ldots, u_{i_m})$\textemdash as a function of the $n-m$ variables $z_{i_{m+1}},\ldots,z_{i_n}$, where $\{i_{m+1},\ldots,i_n\}=\{1,\ldots, n\}\setminus \{i_{1},\ldots,i_m\}$. 
In other words, $f_{\mathbf{w}}$ is a non-degenerate function of the variables $(z_i)_{i\notin I}$, with the other variables $(z_i)_{i\in I}$ being fixed in the ball $\sum_{i\in I}\vert z_{i}\vert^2 < R(f)$. Any positive number $R(f)$ for which the above property holds is called a \emph{radius of local tameness} of $f$.

\subsection{Non-degenerate locally tame complete intersection variety}

Let us now consider $k_0$ non-constant polynomial functions $f^1(\mathbf{z}),\ldots, f^{k_0}(\mathbf{z})$ which all vanish at the origin. We say that the germ at $\mathbf{0}$ of the variety 
\begin{equation*}
V(f^1,\ldots,f^{k_0}):=\{\mathbf{z}\in\mathbb{C}^n\mid f^1(\mathbf{z}) = 
\cdots = f^{k_0}(\mathbf{z})=0\}
\end{equation*}
is a germ of a \emph{non-degenerate} complete intersection variety if for any positive weight vector $\mathbf{w}$, the toric variety
\begin{equation*}
V^*(f^1_{\mathbf{w}},\ldots,f^{k_0}_{\mathbf{w}}):=\{\mathbf{z}\in\mathbb{C}^{*n}\mid f^1_{\mathbf{w}}(\mathbf{z}) = \cdots = f^{k_0}_{\mathbf{w}}(\mathbf{z})=0\}
\end{equation*}
is a reduced, non-singular, complete intersection variety in $\mathbb{C}^{*n}$. This means that the $k_0$-form 
\begin{equation*}
df^1_{\mathbf{w}}\wedge \cdots\wedge df^{k_0}_{\mathbf{w}}
\end{equation*}
is nowhere vanishing in  $V^*(f^1_{\mathbf{w}},\ldots,f^{k_0}_{\mathbf{w}})$. Again, let us emphasize that the variety $V^*(f^1_{\mathbf{w}},\ldots,f^{k_0}_{\mathbf{w}})$ is globally defined in~$\mathbb{C}^{*n}$.

Finally, we say that the germ at $\mathbf{0}$ of $V(f^1,\ldots,f^{k_0})$  is a germ of a \emph{locally tame} complete intersection variety if there is a positive number $R(f^1,\ldots,f^{k_0})$ such that for any non-empty subset $I:=\{i_1,\ldots, i_m\}\in \mathcal{V}_{f^1}\cap\cdots\cap \mathcal{V}_{f^{k_0}}$, any non-zero weight vector $\mathbf{w}$ with $I(\mathbf{w})=I$, and any non-zero complex numbers $u_{i_1},\ldots, u_{i_m}$ satisfying the inequality
\begin{equation*}
\sum_{j=1}^m\vert u_{i_j}\vert^2 < R(f^1,\ldots,f^{k_0}), 
\end{equation*}
the toric variety
\begin{equation*}
V^*(f^1_{\mathbf{w}},\ldots,f^{k_0}_{\mathbf{w}})\cap 
\mathbb{C}^{*n}(u_{i_1},\ldots, u_{i_m})
\end{equation*}
is a reduced, non-singular, complete intersection variety in $\mathbb{C}^{*n}(u_{i_1},\ldots, u_{i_m})$. Again, this means that the $k_0$-form $df^1_{\mathbf{w}}\wedge \cdots\wedge df^{k_0}_{\mathbf{w}}$ is nowhere vanishing in $V^*(f^1_{\mathbf{w}},\ldots,f^{k_0}_{\mathbf{w}})\cap \mathbb{C}^{*n}(u_{i_1},\ldots, u_{i_m})$. Any positive number $R(f^1,\ldots,f^{k_0})$ satisfying the above property is called a \emph{radius of local tameness} of the set of functions $\{f^1,\ldots,f^{k_0}\}$.

\section{Whitney equisingularity}\label{sect-we}

Now, let $(t,\mathbf{z}):=(t,z_1,\ldots, z_n)$ be coordinates for $\mathbb{C}\times \mathbb{C}^n$, and for any $k\in K_0:=\{1,\ldots,k_0\}$, let $f^k \colon \mathbb{C}\times \mathbb{C}^n \rightarrow \mathbb{C}$ be a non-constant polynomial function such that $f^k(t,\mathbf{0})=0$ for all $t$. Put 
\begin{align*}
f(t,\mathbf{z}) := f^1(t,\mathbf{z}) \cdots f^{k_0}(t,\mathbf{z}), 
\end{align*}
and as usual, write $f_t(\mathbf{z}):=f(t,\mathbf{z})$ and $f^k_t(\mathbf{z}):=f^k(t,\mathbf{z})$. 
We suppose that for any sufficiently small $t$, the following two conditions hold true:
\begin{enumerate}
\item[$\cdot$]
for any $k\in K_0$, the Newton boundary $\Gamma(f^k_t)$ is independent of $t$ (especially, $\mathcal{V}_{f_t^k}$ is independent of $t$);
\item[$\cdot$]
for any $\{k_1,\ldots,k_p\}\subseteq K_0$, the germ at $\mathbf{0}$ of $V(f^{k_1}_t,\ldots,f^{k_p}_t)$ is a germ of a non-degenerate complete intersection variety.
\end{enumerate}
Then, by \cite[Chap.~V, Lemma (2.8.2)]{O6}, there is a neighbourhood of the origin in which any subset of the form 
\begin{align}\label{p3-thesubset}
\bigcap_{k\in K} V(f^k)\cap (\mathbb{C}\times \mathbb{C}^{*I})
\end{align}
is non-singular, where $K\subseteq K_0$ and $I\subseteq \{1,\ldots,n\}$. Note that if for all $k\in K$ the restriction $f^k\vert_{\mathbb{C}\times\mathbb{C}^I}$ is not identically zero, then the subset \eqref{p3-thesubset} is also a complete intersection variety (ibid.).
Hereafter, we shall write 
\begin{align*}
V^{*I}(f^k):=V(f^k)\cap (\mathbb{C}\times \mathbb{C}^{*I}). 
\end{align*}
It follows that the collection $\mathcal{S}$ of all non-empty subsets of the form
\begin{align*}
S^I(K) & := \{(t,\mathbf{z})\in\mathbb{C}\times \mathbb{C}^{*I} \mid 
f^k(t,\mathbf{z})=0 \Leftrightarrow k\in K\}\\
& \ = \bigcap_{k\in K} V^{*I}(f^k) \bigg\backslash \bigcup_{k\in K_0\setminus K} V^{*I}(f^k)
\end{align*}
is a complex analytic stratification of $V(f)$. We call $\mathcal{S}$ the \emph{canonical toric stratification} of $V(f)$. Note that it includes $S^{\emptyset}(K_0)=\mathbb{C}\times\{\mathbf{0}\}$ (i.e., the $t$-axis) as a stratum. 

\begin{remark}
For a given $I$, if $S^I(K)\not=\emptyset$, then $K$ necessarily contains all $k$'s for which $f^k\vert_{\mathbb{C}\times\mathbb{C}^I}\equiv 0$ plus possibly some other $k$'s such that $f^k\vert_{\mathbb{C}\times\mathbb{C}^I}\not\equiv 0$.
\end{remark}

To state the results of this paper, we need to strengthen a bit the assumptions made above. Precisely we fix the following terminology.

\begin{definition}\label{maindef}
We say that the family $\{f_t\}_t$ is \emph{Newton-admissible} if for any sufficiently small $t$, the following two conditions hold true:
\begin{enumerate}
\item[$\cdot$]
for any $k\in K_0$, the Newton boundary $\Gamma(f^k_t)$ is independent of $t$;
\item[$\cdot$]
for any $\{k_1,\ldots,k_p\}\subseteq K_0$, the germ at $\mathbf{0}$ of $V(f^{k_1}_t,\ldots,f^{k_p}_t)$ is a germ of a non-degenerate, locally tame, complete intersection variety, and there is a radius of local tameness $R(f^{k_1}_t,\ldots,f^{k_p}_t)$ for the corresponding set of functions $\{f^{k_1}_t,\ldots,f^{k_p}_t\}$ which is greater than some positive number $R$ independent of $t$ and of the choice of the subset $\{k_1,\ldots,k_p\}$.
\end{enumerate}
\end{definition}

The inequality $R(f^{k_1}_t,\ldots,f^{k_p}_t)>R$, which holds for all small $t$, expresses the fact that the family $\{V(f^{k_1}_t,\ldots,f^{k_p}_t)\}_t$ is \emph{uniformly} locally tame.

The main result of this section is stated as follows.

\begin{theorem}\label{mt}
If the family $\{f_t\}_t$ is Newton-admissible, then the canonical toric stratification $\mathcal{S}$ of $V(f)$ is Whitney $(b)$-regular. In particular, the corresponding family of hypersurfaces $\{V(f_t)\}_t$ is Whitney equisingular. 
\end{theorem}

Combined with the Thom--Mather first isotopy theorem (see \cite{T,Ma}), Theorem \ref{mt} shows that if $\{f_t\}_t$ is a Newton-admissible family, then the local, ambient, topological type of $V(f_t)$ at $\mathbf{0}$ is independent of $t$ for all small $t$.

Theorem \ref{mt} is proved in Section \ref{sect-proofmt}.
It extends our previous result \cite[Theorem 3.8]{EO} to a much larger class of hypersurfaces with non-isolated singularities. 

The Whitney $(b)$-regularity along the $t$-axis for strata of the form $S^J(L)$, $J\notin\mathcal{V}_{f^k_t}$, $k\in L$, is already obtained in \cite[Chap.~V, Theorem (2.8)]{O6}. However, this is not enough to ensure that the family $\{V(f_t)\}_t$ is Whitney equisingular (or even topologically equisingular). We also need that pairs of strata $S^I(K)\subseteq\overline{S^J(L)}$, with $I\in\mathcal{V}_{f^k_t}$ for $k\in K$ and $J\notin\mathcal{V}_{f^k_t}$ for $k\in L$, satisfy the  Whitney $(b)$-regularity condition.
On the other hand, if we only ask for Whitney's $(b)$-regularity only along the $t$-axis for strata of the form $S^J(L)$, $J\notin\mathcal{V}_{f^k_t}$, $k\in L$, then the uniform local tameness assumption is not needed; only the stability of the Newton boundary and the non-degeneracy condition are required.

The topological equisingularity can be also obtained (it is done in \cite[Chap.~V, Corollary (2.11)]{O6}) under the stability of the Newton boundary, the non-degeneracy condition, and an additional assumption which guarantees the independence in $t$ of the face functions $f^1_{t,\mathbf{w}},\ldots,f^{k_0}_{t,\mathbf{w}}$ of $f^1_{t},\ldots,f^{k_0}_{t}$ with respect to so-called ``essential'' weight vectors $\mathbf{w}$ (here, ``essential'' means $\emptyset\not=I(\mathbf{w})\in \mathcal{V}_{f^1_t}\cap\cdots\cap \mathcal{V}_{f^{k_0}_t}$ and $V(f^1_{t,\mathbf{w}},\ldots,f^{k_0}_{t,\mathbf{w}})\not=\emptyset$).

In \cite[\S 8]{O}, another Whitney stratification of the pair $(V(f),\mathbb{C}\times \{\mathbf{0}\})$\textemdash with a larger number of strata than ours\textemdash is constructed under a different, rather technical, assumption (so-called ``simultaneous IND-condition''). Under this condition, both Whitney and topological equisingularities do hold too.

The following is an immediate corollary of Theorem \ref{mt}.

\begin{corollary}\label{cormt}
If the family $\{f_t\}_t$ is Newton-admissible, then for any sufficiently small $t$, the partition $\mathcal{S}_t=\{S^I_t(K)\}_{I,K}$  of $\{t\}\times\mathbb{C}^n$, defined by 
\begin{equation*}
S^I_t(K):=S^I(K)\cap (\{t\}\times\mathbb{C}^n)
\end{equation*}
where $I\subseteq \{1,\ldots,n\}$ and $K\subseteq K_0$, is a Whitney $(b)$-regular stratification of $V(f_t)$ in a neighbourhood of the origin of $\mathbb{C}^n$ which is independent of $t$. (Here, we identify $\{t\}\times\mathbb{C}^n$ with $\mathbb{C}^n$.) More precisely, there exist an open disc $\mathring{D}_\tau\subseteq \mathbb{C}$ with radius $\tau>0$ and an open ball $\mathring{B}_r\subseteq\mathbb{C}^n$ with radius $r>0$ centred at the origins of $\mathbb{C}$ and $\mathbb{C}^n$, respectively, such that for any $t\in \mathring{D}_\tau$, the collection $\{S^I_t(K)\cap \mathring{B}_r\}_{I,K}$ is a Whitney $(b)$-regular stratification of $V(f_t)\cap \mathring{B}_r$.
\end{corollary}

Indeed, if $\{f_t\}_t$ is Newton-admissible, then there exist $\tau,r>0$ such that for any $t\in \mathring{D}_\tau$, the space $\{t\}\times\mathbb{C}^n$ transversely intersects all the strata of $\mathcal{S}$ in the neighbourhood $\mathring{D}_\tau\times \mathring{B}_r$ of the origin of $\mathbb{C}\times \mathbb{C}^n$, and it is well known that Whitney $(b)$-regular stratifications are preserved under transverse intersections (see \cite{GWPL}).

\begin{remark}
If the family $\{f_t\}_t$ is Newton-admissible, then Theorem \ref{mt} shows that the collection of subsets $\{S^I(K_0)\}_{I\subseteq \{1,\ldots,n\}}$ is a Whitney $(b)$-regular stratification of $V(f^1,\ldots,f^{k_0})$ with the $t$-axis as a stratum. In particular, by the Thom--Mather first isotopy theorem again, the local, ambient, topological type at $\mathbf{0}$ of the complete intersection variety $V(f_t^1,\ldots,f_t^{k_0})$ is independent of $t$ for all small $t$.

Similarly, Corollary \ref{cormt} shows that if $\{f_t\}_t$ is Newton-admissible, then for any sufficiently small $t$, the collection of subsets $\{S^I(K_0)\cap (\{t\}\times\mathbb{C}^n)\}_{I\subseteq \{1,\ldots,n\}}$ is a Whitney $(b)$-regular stratification of the complete intersection variety $V(f_t^1,\ldots,f_t^{k_0})$ in a neighbourhood of the origin of $\mathbb{C}^n$ which is independent of $t$.
\end{remark}

Theorem \ref{mt} has several other important corollaries. They are stated in the following two sections.

\section{Thom's $a_f$ condition}\label{sect-tafc}

Pick a sufficiently small representative of (the germ at $\mathbf{0}$ of) $f$ so that $0$ is the only possible critical value of $f$. Then the critical locus $\Sigma f$ of $f$ is contained in $V(f)$.
We say that a Whitney $(a)$-regular stratification of $V(f)$ satisfies \emph{Thom's $a_f$ condition} if for any stratum $S$, any point  $(\tau,\mathbf{q})\in S$, and any sequence $\{(\tau_m,\mathbf{q}_m)\}_m\notin V(f)$ such that
\begin{equation*}
(\tau_m,\mathbf{q}_m)\to (\tau,\mathbf{q})
\quad\mbox{and}\quad
T_{(\tau_m,\mathbf{q}_m)}V(f-f(\tau_m,\mathbf{q}_m))\to T
\end{equation*}
as $m\to\infty$, we have $T_{(\tau,\mathbf{q})}S\subseteq T$,
where 
\begin{equation*}
V(f-f(\tau_m,\mathbf{q}_m)):=\{(t,\mathbf{z})\in \mathbb{C}\times\mathbb{C}^n \mid f(t,\mathbf{z})=f(\tau_m,\mathbf{q}_m)\}
\end{equation*}
and $T_{(\tau_m,\mathbf{q}_m)}V(f-f(\tau_m,\mathbf{q}_m))$ is the tangent space of $V(f-f(\tau_m,\mathbf{q}_m))$ at $(\tau_m,\mathbf{q}_m)$. Similarly, $T_{(\tau,\mathbf{q})}S$ is the tangent space of $S$ at $(\tau,\mathbf{q})$.

Combined with a result of Parusi\'nski \cite{P} and Brian\c con--Maisonobe--Merle \cite{BMM}, Theorem \ref{mt} above implies the following statement.

\begin{theorem}\label{mt2}
If the family $\{f_t\}_t$ is Newton-admissible, then the canonical toric stratification $\mathcal{S}$ of $V(f)$ satisfies Thom's $a_f$ condition.
\end{theorem}

Indeed, by Theorem \ref{mt}, the stratification $\mathcal{S}$ is Whitney $(b)$-regular. Then Theorem p.~99 of \cite{P} or Theorem 4.3.2 of \cite{BMM} tell us that $\mathcal{S}$ satisfies so-called $w_f$ condition, and hence, Thom's $a_f$ condition.
Theorem \ref{mt2} can also be proved by a direct explicit calculation, without invoking \cite{P} or \cite{BMM} (see Section~\ref{sect-proofmt}). 
We thank L\^e D\~ung Tr\'ang who pointed out us the result of Parusi\'nski, Brian\c con, Maisonobe and Merle, which considerably simplifies our original direct proof.

By a straightforward extension of \cite[Theorem 5]{O3}, if for any subset $\{k_1,\ldots,k_p\}\subseteq K_0$, the germ at $(0,\mathbf{0})\in\mathbb{C}\times\mathbb{C}^n$ of the variety $V(f^{k_1},\ldots,f^{k_p})$ is a germ of a non-degenerate, locally tame, complete intersection variety, then the canonical toric stratification of $V(f)$ satisfies Thom's $a_f$ condition too. However, note that in this latter statement, the assumptions refer to the functions $f^{k_j}(t,\mathbf{z})$ whereas in Theorem \ref{mt2} they refer to the members $f^{k_j}_t(\mathbf{z})$ of the families $\{f^{k_j}_t\}_t$ defined by the functions $f^{k_j}(t,\mathbf{z})$.

The following proposition, which asserts the ``uniform'' smoothness of the nearby fibres of a Newton-admissible family, will be very useful.

\begin{proposition}\label{proposition-us}
If the family $\{f_t\}_t$ is Newton-admissible, then the nearby fibres $V(f_t-\eta)$, $\eta\not=0$, of the functions $f_t$ are ``uniformly'' non-singular with respect to the deformation parameter $t$. That is, there exist positive numbers $\delta,\tau,r$ such that for any $0<\vert\eta\vert\leq\delta$ and any $0\leq\vert t\vert\leq \tau$, the level hypersurface $V(f_t-\eta)$ is non-singular in $\mathring{B}_r$. (Here, $\mathring{B}_r$ denotes the open ball with radius $r$ centred at the origin.)
\end{proposition}

The proof of Proposition \ref{proposition-us} is given in Section \ref{proofofproposition-us}. In particular, this proposition implies that, in $\mathring{B}_r$, the critical locus $\Sigma f_t$ of $f_t$ is contained in $V(f_t)$ for any $\vert t\vert\leq \tau$, and hence for any $\mathbf{z}\in \mathring{B}_r\setminus \Sigma f_t$ the tangent space $T_{\mathbf{z}}V(f_t-f_t(\mathbf{z}))$ exists. We can thus state the following corollary of  Theorem \ref{mt2}.

\begin{corollary}\label{cor-mt2}
If the family $\{f_t\}_t$ is Newton-admissible, then the stratifications $\mathcal{S}_t$ of $V(f_t)$ defined in Corollary \ref{cormt} satisfy Thom's $a_{f_t}$ condition ``uniformly''  with respect to the parameter $t$, provided that the latter is small enough. More precisely, there exist an open disc $\mathring{D}_\tau\subseteq \mathbb{C}$ with radius $\tau>0$ and an open ball $\mathring{B}_r\subseteq\mathbb{C}^n$ with radius $r>0$ centred at the origins of $\mathbb{C}$ and $\mathbb{C}^n$, respectively, such that for any $t\in \mathring{D}_\tau$, any stratum $S\in \mathcal{S}_t$, any point $\mathbf{q}\in S$, and any sequence $\{\mathbf{q}_m\}_m\in \mathring{B}_r\setminus V(f_t)$ such that
\begin{equation*}
\mathbf{q}_m\to \mathbf{q}
\quad\mbox{and}\quad
T_{\mathbf{q}_m}V(f_t-f_t(\mathbf{q}_m))\to T
\end{equation*}
as $m\to\infty$, we have $T_{\mathbf{q}}S\subseteq T$.
\end{corollary}

\section{Milnor fibrations}\label{sect-MF}

The main result of this section is Theorem \ref{mt4}. It asserts that if the family $\{f_t\}_t$ is Newton-admissible, then the Milnor fibrations of $f_t$ and $f_0$ at $\mathbf{0}$ are isomorphic. A key ingredient of the proof is Theorem \ref{mt3} which says that Newton-admissible families have a ``uniform stable radius.''

Before stating the results, let us then recall the definition of ``uniform stable radius.'' By a result of Hamm and L\^e D\~ung Tr\'ang \cite[Lemme (2.1.4)]{HL}, we know that for each $t$ there exists a positive number $r_t>0$ such that for any pair $(\varepsilon_t,\varepsilon'_t)$ with $0<\varepsilon'_t\leq \varepsilon_t\leq r_t$, there exists $\delta(\varepsilon_t,\varepsilon'_t)>0$ such that for any non-zero complex number $\eta$ with $0<\vert\eta\vert\leq\delta(\varepsilon_t,\varepsilon'_t)$, the hypersurface $V(f_t-\eta)$ is non-singular in $\mathring{B}_{r_t}$ and transversely intersects the sphere $\mathbb{S}_{\varepsilon''}$ for any $ \varepsilon''$ with $\varepsilon'_t\leq \varepsilon''\leq \varepsilon_t$. Any such a number $r_t$ is called a \emph{stable radius} for the Milnor fibration of $f_t$ at $\mathbf{0}$. Here, $\mathring{B}_{\rho}\subseteq \mathbb{C}^n$ is the open ball with radius $\rho$ centred at $\mathbf{0}$ and $\mathbb{S}_{\rho}$ denotes the boundary of the corresponding closed ball $B_{\rho}$.
Now, following \cite[\S 3]{O}, we say that the family $\{f_t\}_t$ has a \emph{uniform stable radius} if there exist $\tau > 0$ and $r > 0$ such that for any pair $(\varepsilon,\varepsilon')$ with $0<\varepsilon'\leq \varepsilon\leq r$, there exists $\delta(\varepsilon,\varepsilon')>0$ such that for any non-zero complex number $\eta$ with $0<\vert\eta\vert\leq\delta(\varepsilon,\varepsilon')$, the hypersurface $V(f_t-\eta)$ is non-singular in $\mathring{B}_{r}$ and transversely intersects the sphere $\mathbb{S}_{\varepsilon''}$ for any $ \varepsilon''$ with $\varepsilon'\leq \varepsilon''\leq \varepsilon$ and for any $t$ in the closed disc $D_\tau:=\{t\in\mathbb{C}\mid 0\leq \vert t\vert \leq\tau\}$.
Any such a number $r$ is called a \emph{uniform stable radius} for $\{f_t\}_t$. 

\begin{theorem}\label{mt3}
If the family $\{f_t\}_t$ is Newton-admissible, then it has a uniform stable radius.
\end{theorem}

Theorem \ref{mt3} is proved in Section \ref{proofmt3}. The proof is based on the following proposition (proved in Section \ref{proofprop-utr}) which is interesting itself.
It is well known that for a given $t$, small spheres are transverse to the strata of $\mathcal{S}_t$. The proposition says that, under the Newton-admissibility condition,  this property holds ``uniformly'' in $t$. The precise statement is as follows.

\begin{proposition}\label{lemma-fpofmt3}
If the family $\{f_t\}_t$ is Newton-admissible, then there exists a neighbourhood $\mathring{D}_\tau\times\mathring{B}_r$ of the origin of $\mathbb{C}\times \mathbb{C}^n$ such that for any $0<r'\leq r$ and any $t\in\mathring{D}_\tau$, the sphere $\{t\}\times \mathbb{S}_{r'}$ transversely intersects $(\{t\}\times\mathbb{C}^n)\cap S$ for any stratum $S\in\mathcal{S}$, where $\mathcal{S}$ is the stratification constructed in Theorem \ref{mt} and $\mathbb{S}_{r'}$ is the sphere with radius $r'$ centred at the origin of $\mathbb{C}^n$.
\end{proposition}

Combined with \cite[Lemma 2]{O82}, Theorem \ref{mt3} implies the main result of this section the statement of which is as follows.

\begin{theorem}\label{mt4}
If the family $\{f_t\}_t$ is Newton-admissible, then the Milnor fibrations of $f_t$ and $f_0$ at $\mathbf{0}$ are isomorphic for all small $t$.
\end{theorem}

The rest of the paper is devoted to the proofs of Theorems \ref{mt} and \ref{mt3} (we also give a direct proof of Theorem \ref{mt2}) and to the proofs of Propositions \ref{proposition-us} and \ref{lemma-fpofmt3}.

\section{Proof of Theorem \ref{mt}}\label{sect-proofmt}

It is inspired from the proofs of \cite[Theorem (5.1)]{O}, \cite[Chap.~V, Theorem (2.8)]{O6}, \cite[Theorem 3.8]{EO} and \cite[Theorem 3.14]{EO2}. 

We shall use the following terminology. A non-singular holomorphic $k$-form $\omega$ is called \emph{decomposable} at $(t,\mathbf{z})$ if there exist linearly independent holomorphic $1$-forms $\omega_1,\ldots,\omega_{k}$ in a neighbourhood of $(t,\mathbf{z})$ such that 
\begin{equation*}
\omega(t,\mathbf{z})=\omega_1(t,\mathbf{z})\wedge\cdots\wedge\omega_k(t,\mathbf{z}). 
\end{equation*}
To such a form, we can associate an $(n+1-k)$-dimensional subspace $\omega(t,\mathbf{z})^\bot$ of the tangent space $T_{(t,\mathbf{z})}(\mathbb{C}\times \mathbb{C}^n)$ defined by
\begin{equation*}
\omega(t,\mathbf{z})^\bot := \{\mathbf{v}\in T_{(t,\mathbf{z})}(\mathbb{C}\times \mathbb{C}^n) \mid \imath_{\mathbf{v}}(\omega(t,\mathbf{z}))=0\},
\end{equation*}
where $\imath_{\mathbf{v}}(\omega(t,\mathbf{z}))$ is the inner derivative of $\omega(t,\mathbf{z})$ by $\mathbf{v}$.

\subsection{Plan of the proof}
To prove that our stratification $\mathcal{S}$ is Whitney $(b)$-regular, it suffices to show that for any $I\subseteq J\subseteq\{1,\ldots,n\}$ and any $L\subseteq K\subseteq K_0$ with $S^I(K)\cap\overline{S^J(L)}\not=\emptyset$, the ``big'' stratum $S^J(L)$ is Whitney $(b)$-regular over the ``small'' stratum $S^I(K)$ at any point 
\begin{align*}
(\tau,\mathbf{q})=(\tau,q_1,\ldots,q_n)\in S^I(K)\cap \overline{S^J(L)}
\end{align*} 
 sufficiently close to the origin. Here, $\overline{S^J(L)}$ denotes the closure of $S^J(L)$. To simplify, we shall assume:
\begin{align*}
\emptyset\not=I=\{1,\ldots,n_I\} \subseteq J= \{1,\ldots,n\}
\quad\mbox{and}\quad
L=\{1,\ldots,k_L\} \subseteq K.
\end{align*} 
(For the case $I=\emptyset$, see \S \ref{conclusion}.)
In particular, since $(\tau,\mathbf{q})\in S^I(K)$, we have $q_i\not=0$ if and only if $1\leq i\leq n_I$. 

Pick real analytic paths 
\begin{align*}
\rho(s):=(t(s),\mathbf{z}(s))
\quad\mbox{and}\quad
\rho'(s):=(t'(s),{\mathbf{z}}'(s))
\end{align*} 
such that:
\begin{enumerate}
\item
$\rho(0)=\rho'(0)=(\tau,\mathbf{q})$;
\item
$\rho'(s)\in S^I(K)$ and $\rho(s)\in S^J(L)$ for $s\not=0$.
\end{enumerate}
Write $\mathbf{z}(s)=(z_1(s),\ldots,z_n(s))$ and $\mathbf{z}'(s)=(z'_1(s),\ldots,z'_n(s))$, and look at the Taylor expansions:
\begin{align*}
& t(s)=t_0 s^{w_0}+b_0s^{w_0+1}+\cdots,\quad
z_i(s)=a_i s^{w_i}+b_i s^{w_i+1}+\cdots,\\
& t'(s)=t'_0 s^{w'_0}+b'_0 s^{w'_0+1}+\cdots, \quad
z'_i(s)=q_i +b'_i s+\cdots,
\end{align*}
where $w_i=0$ and $a_i=q_i$ for $1\leq i\leq n_I$ while $w_i>0$ and $a_i\not=0$ for $i>n_I$. The coefficients $t_0$ and $t'_0$ are also assumed to be non-zero. Note that if $\tau\not=0$, then $w_0=w'_0=0$ and $t_0=t(0)=\tau=t'(0)=t'_0$; if $\tau=0$, then $w_0,w'_0>0$ and $t(0)=\tau=t'(0)=0$. Throughout, the dots ``$\cdots$'' stand for the terms of higher degree in $s$. Note that if $i>n_I$, then $z'_i(s)=0$ for any $s$. Hereafter, we shall write $\mathbf{w}:=(w_1,\ldots,w_n)$ and $\mathbf{a}:=(a_1,\ldots,a_n)$. Put
\begin{equation*}
\ell(s):=(\ell_0(s),\ell_1(s),\ldots,\ell_n(s)),
\end{equation*}
where
\begin{equation*}
\left\{
\begin{aligned}
& \ell_0(s):=t(s)-t'(s)\\
& \ell_i(s):=z_i(s)-z'_i(s)=\left\{
\begin{aligned}
& (b_i-b'_i)s+\cdots && \mbox{ for}\quad 1\leq i\leq n_I,\\
& a_i s^{w_i}+\cdots && \mbox{ for}\quad n_I+1\leq i\leq n.
\end{aligned}
\right.
\end{aligned}
\right.
\end{equation*}
By reordering, we may suppose $w_{n_I+1}\leq w_{n_I+2}\leq \cdots\leq w_{n}$.
We shall write 
\begin{equation*}
w_{\mbox{\tiny min}}:=w_{n_I+1}=\cdots=w_{n_I+m_{I}}
\end{equation*}
($m_I\geq 1$) for the minimal value of these weights.

Let $o_\ell:=\mbox{min} \, \{\mbox{ord}\, \ell_i(s) \mid 0\leq i\leq n\}$,
where $\mbox{ord}\, \ell_i(s)$ is the order in $s$ of the $i$th component $\ell_i(s)$ of $\ell(s)$, and let
\begin{equation*}
\ell_\infty:=\lim_{s\to 0} \frac{\ell(s)}{\Vert \ell(s) \Vert} 
=\lim_{s\to 0}\frac{\ell(s)}{\vert s\vert^{o_\ell}}
\quad\mbox{and}\quad
T_\infty:=\lim_{s\to 0} T_{\rho(s)} S^J(L).
\end{equation*}
Here, $T_{\rho(s)} S^J(L)$ denotes the (complex) tangent space of $S^J(L)$ at $\rho(s)$, which is given by
\begin{align*}
T_{\rho(s)} S^J(L) & = (df^1(\rho(s))\wedge \cdots\wedge df^{k_L}(\rho(s)))^\bot\\
& = \{ \mathbf{v}\in T_{\rho(s)} (\mathbb{C}\times \mathbb{C}^n)\mid \imath_{\mathbf{v}}(df^1(\rho(s))\wedge \cdots\wedge df^{k_L}(\rho(s)))=0\}\\
& = \{ \mathbf{v}\in T_{\rho(s)} (\mathbb{C}\times \mathbb{C}^n)\mid df^k(\rho(s))(\mathbf{v})=0 \mbox{ for all } 1\leq k\leq k_L\}\\
& = \bigcap_{1\leq k\leq k_L} df^k(\rho(s))^\bot,
\end{align*}
where 
\begin{align*}
df^k(\rho(s))=\frac{\partial f^k}{\partial t}(\rho(s))\, dt + 
\sum_{i=1}^n \frac{\partial f^k}{\partial z_i}(\rho(s))\, dz_i.
\end{align*}
(This follows from \cite[Chap.~V, Lemma (2.8.2)]{O6}, which says that $S^J(L)$ is a non-singular complete intersection variety at $\rho(s)$.)

By \cite[Proposition (2.2)]{O}, to show that $S^J(L)$ is Whitney $(b)$-regular over $S^I(K)$ at the point $(\tau,\mathbf{q})$, it suffices to prove that 
\begin{equation}\label{lfwbrc}
\ell_\infty \in T_\infty.
\end{equation}
Write $o_{k_L}$ for the order in $s$ of $df^1(\rho(s))\wedge \cdots\wedge df^{k_L}(\rho(s))$, and consider
\begin{equation*}
\omega_\infty:=\lim_{s\to 0} \frac{1}{s^{o_{k_L}}} \cdot df^1(\rho(s))\wedge \cdots\wedge df^{k_L}(\rho(s)) \in\wedge^{k_L} T^*_{(\tau,\mathbf{q})}(\mathbb{C}\times \mathbb{C}^n),
\end{equation*}
where $T^*_{(\tau,\mathbf{q})}(\mathbb{C}\times \mathbb{C}^n)$ denotes the cotangent space. Clearly, 
\begin{equation*}
T_\infty=\omega_\infty^\bot, 
\end{equation*}
and to prove \eqref{lfwbrc}, we must show that  $\ell_\infty\in \omega_\infty^\bot$. 
For that purpose, we shall prove that $\omega_\infty$ is decomposable, that is, there exist linearly independent $1$-forms $\omega_1,\ldots,\omega_{k_L}$ such that $\omega_\infty=\omega_1\wedge\cdots\wedge\omega_{k_L}$. Therefore, to prove that $\ell_\infty\in\omega_\infty^\bot$, it will be enough to show that $\omega_k(\ell_\infty)=0$ for all $1\leq k\leq k_L$.
Here, the main difficulty is that, in general, if $\mathbb{C}^I$ is a vanishing coordinate subspace for some of the functions $f^k$, then the limits of the $1$-forms $df^1(\rho(s)),\ldots,df^{k_L}(\rho(s))$ as $s\to 0$ are not linearly independent. We shall solve this problem using the same subterfuge as in \cite{O,O6,O2,O3,EO2} and substituting to the corresponding differential $df^k(\rho(s))$ a term of the form
\begin{equation*}
df^k(\rho(s))+\sum_{l=1}^{k-1}c_{k,l}(s)\, df^{l}(\rho(s)),
\end{equation*}
where $c_{k,l}(s)$ are suitable polynomials.

The plan of the proof is as follows. 
First, in \S \ref{proof-para1}, using the uniform local tameness condition and the subterfuge mentioned above, we construct the forms $\omega_k$ corresponding to the functions $f^k$ having $\mathbb{C}^I$ as a vanishing coordinate subspace (i.e., $I\in\mathcal{V}_{f^k_t}$). Then, using the non-degeneracy condition, we construct the forms $\omega_k$ corresponding to the functions $f^k$ such that $I\notin\mathcal{V}_{f^k_t}$ (see \S \ref{proof-para2}). In \S \ref{synthesis}, we combine the results of \S\S \ref{proof-para1} and \ref{proof-para2} to deduce that $\omega_\infty=\omega_1\wedge\cdots\wedge\omega_{k_L}$. Finally, in \S \ref{76} (respectively, in \S \ref{77}), we show that $\omega_k(\ell_\infty)=0$ for the indices $k$ such that $I\in\mathcal{V}_{f^k_t}$ (respectively, $I\notin\mathcal{V}_{f^k_t}$).

\subsection{Construction of the forms $\omega_k$ corresponding to the functions $f^k$ such that $I\in\mathcal{V}_{f^k_t}$}\label{proof-para1}
Let $L'\subseteq L=\{1,\ldots,k_L\}$ be the subset of all indices $k$ for which $\mathbb{C}^I$ is a vanishing coordinate subspace for $f_t^k$, that is,
\begin{equation*}
L':=\{k\in L\mid I\in\mathcal{V}_{f^k_t}\}.
\end{equation*}
(We remind that $\mathcal{V}_{f_t^k}$ is independent of $t$.)
To simplify, we shall assume that $L'=\{1,\ldots,k_{L'}\}$. 
By multiplying suitable monomials in the variables $(z_i)_{i\in J\setminus I}$ to the functions $f^k_t(\mathbf{z})$ if necessary, we may also assume that 
\begin{equation*}
d(\mathbf{w};f^1_t)=\cdots=d(\mathbf{w};f^{k_{L'}}_t),
\end{equation*}
as such an operation does not change anything to the stratum $S^J(L)$. Let us denote by $d$ this common value which, by the constancy of the Newton boundary, does not depend on $t$. Then we have the following lemma. Its statement is similar to \cite[Lemma (5.11)]{O} and \cite[Chap.~V, Lemma~(2.8.19)]{O6}. However its assumptions, which are those of Theorem \ref{mt}, are different (see the comments that follow the statement of Theorem \ref{mt}). 

\begin{lemma}\label{lemma-fundlil}
After renumbering the functions $f^1,\ldots, f^{k_{L'}}$ if necessary, we can find polynomials $c_{k,l}(s)$ for $1\leq l<k\leq k_{L'}$ such that if we write
\begin{equation}\label{egalitedulemme}
df^k(\rho(s))+\sum_{l=1}^{k-1}c_{k,l}(s)\, df^{l}(\rho(s)) =
\omega_k s^{\nu_k} + \cdots,
\end{equation}
with $\omega_{k}\not=0$ for $1\leq k\leq k_{L'}$, then $\omega_1,\ldots,\omega_{k_{L'}}$ are linearly independent cotangent vectors in $T^*_{(\tau,\mathbf{q})}(\mathbb{C}\times \mathbb{C}^n)$ and $\nu_1\leq\cdots \leq \nu_{k_{L'}}\leq d-w_{\mbox{\tiny \emph{min}}}$. 
\end{lemma}

\begin{proof}
It is on the same pattern as the proofs of \cite[Lemma (5.11)]{O} and \cite[Chap.~V, Lemma (2.8.19)]{O6}. We divide it into two steps.

\vskip 2mm\noindent
\emph{First step.}
Take any $1\leq k\leq k_{L'}$, and write $f^k_{\tau,\mathbf{w}}$ (respectively, $f^k_{\hat{\mathbf{w}}}$) for the face function of $f^k_\tau$ (respectively, of $f^k$) with respect to the weight vector $\mathbf{w}=(w_1,\ldots,w_n)$ (respectively, the weight vector $\hat{\mathbf{w}}:=(w_0,\mathbf{w})$). 
Since $\Gamma(f^k_t)$ is independent of $t$, we have
\begin{equation}\label{f1-w}
\left\{
\begin{aligned}
& \frac{\partial f^k}{\partial t} (\rho(s)) = \frac{\partial f^k_{\hat{\mathbf{w}}}}{\partial t}(t_0,\mathbf{a})\, s^{d}+\cdots,\\
& \frac{\partial f^k}{\partial z_i} (\rho(s)) = \frac{\partial {f^k_{\tau,\mathbf{w}}}}{\partial z_i} (\mathbf{a})\,  s^{d-w_i}+\cdots,
\end{aligned}
\right.
\end{equation}
so that if $L''\subseteq L'\setminus \{k\}$ and $c_l(s)$, $l\in L''$, are arbitrary polynomials, then
\begin{equation*}
\begin{aligned}
\varphi^{k}(s) := \ & df^{k}(\rho(s))+\sum_{l\in L''}c_l(s)\, df^l(\rho(s))\\
= \ & \biggl(\biggl(\frac{\partial f^k_{\hat{\mathbf{w}}}}{\partial t}(t_0,\mathbf{a}) + \sum_{l\in L''}c_l(0)\, \frac{\partial f^l_{\hat{\mathbf{w}}}}{\partial t}(t_0,\mathbf{a})\biggr)s^d+\cdots\biggr) dt +\\
& \sum_{i=1}^n\biggl(\biggl(\frac{\partial {f^{k}_{\tau,\mathbf{w}}}}{\partial z_i} (\mathbf{a}) 
+\sum_{l\in L''}c_l(0)\, \frac{\partial {f^l_{\tau,\mathbf{w}}}}{\partial z_i} (\mathbf{a})\biggr)s^{d-w_i}+\cdots\biggr) dz_i.
\end{aligned}
\end{equation*}

\begin{claim}\label{expr-varphik1}
In the above expression, at least one of the coefficients
\begin{equation*}
\frac{\partial {f^{k}_{\tau,\mathbf{w}}}}{\partial z_i} (\mathbf{a})+\sum_{l\in L''}c_l(0)\, \frac{\partial {f^l_{\tau,\mathbf{w}}}}{\partial z_i} (\mathbf{a}), 
\end{equation*}
with $i\notin I$, is non-zero. 
\end{claim}

In particular, Claim \ref{expr-varphik1} implies that
\begin{equation*}
\nu_k:=\mbox{ord}(\varphi^{k}(s))\leq d-w_{\mbox{\tiny min}}
\quad\mbox{and}\quad
\omega_{k}:=\lim_{s\to 0}\frac{\varphi^k(s)}{s^{\nu_k}}\in T^*_{(\tau,\mathbf{q})}(\mathbb{C}\times \mathbb{C}^n).
\end{equation*}

\begin{proof}[Proof of Claim \ref{expr-varphik1}]
Since $I\in\mathcal{V}_{f_t^j}$ for all $j\in L'=\{1,\ldots,k_{L'}\}$, the uniform local tameness of the family 
\begin{equation*}
\{V(f_t^1,\ldots,f_t^{k_{L'}})\}_t,
\end{equation*}
tells us that if $(\tau,\mathbf{q})$ is close enough to the origin, then there exist indices $i_{0,1},\ldots,i_{0,k_{L'}}\notin I$ such that the determinant of the $(k_{L'}\times k_{L'})$-matrix
\begin{equation}\label{detnn}
\bigg(\frac{\partial f^j_{\tau,\mathbf{w}}}{\partial z_{i_{0,j'}}} 
(\mathbf{a})\bigg)_{1\leq j,j'\leq k_{L'}}
\end{equation}
does not vanish. Indeed, pick a positive number $R$ as in Definition \ref{maindef} and choose $(\tau,\mathbf{q})$ so small that a radius of local tameness of the set of functions $\{f_\tau^1,\ldots,f_\tau^{k_{L'}}\}$ is greater than $R$.
Since $\rho(s)\in \overline{S^J(L)}$, the expression
\begin{equation*}
f^j(\rho(s))=f^j_{\tau,\mathbf{w}}(\mathbf{a})\, s^d+\cdots
\end{equation*}
vanishes for any $s$ and any $1\leq j\leq k_{L'}$. Clearly, this implies that $\mathbf{a}$ belongs to the toric variety
\begin{equation*}
V^*(f_{\tau,\mathbf{w}}^1,\ldots,f_{\tau,\mathbf{w}}^{k_{L'}}) \cap
\mathbb{C}^{*n}(q_1,\ldots,q_{n_I}), 
\end{equation*}
and since $\mathbb{C}^I$ is a vanishing coordinate subspace for the functions $f_\tau^1,\ldots,f_\tau^{k_{L'}}$, the local tameness condition says that this variety
is a reduced, non-singular, complete intersection variety in $\mathbb{C}^{*n}(q_1,\ldots,q_{n_I})$. In particular, $k_{L'}\leq n-n_I$~and
\begin{equation*}
df^1_{\tau,\mathbf{w}}(\mathbf{a})\wedge\cdots \wedge df^{k_{L'}}_{\tau,\mathbf{w}}(\mathbf{a})\not=0.
\end{equation*}
In other words, 
\begin{equation*}
\sum_{n_I< i_1<\cdots<i_{k_{L'}}\leq n} \det \bigg(\frac{\partial f^j_{\tau,\mathbf{w}}}{\partial z_{i_{j'}}}(\mathbf{a})\bigg)_{1\leq j,j'\leq k_{L'}} \, dz_{i_1}\wedge\cdots\wedge dz_{i_{k_{L'}}}\not=0.
\end{equation*}
Therefore, there exist indices $i_{0,1},\ldots,i_{0,k_{L'}}\notin I=\{1,\ldots,n_I\}$ such that the determinant of the matrix \eqref{detnn} does not vanish as desired.

Now, since we can substitute $\varphi^{k}(s)$ to $df^{k}(\rho(s))$ in the expression 
\begin{equation*}
df^1(\rho(s))\wedge\cdots \wedge df^{k_{L'}}(\rho(s))
\end{equation*}
without changing anything, by comparing the coefficients of $dz_{i_{0,1}}\wedge\cdots\wedge dz_{i_{0,k_{L'}}}$ before and after the substitution, we get the property stated in Claim \ref{expr-varphik1}.
\end{proof}

\noindent
\emph{Second step.}
Now, pick $1\leq k_1\leq k_{L'}$ such that 
\begin{equation*}
\mbox{ord}\, df^{k_1}(\rho(s)) = 
\mbox{min}\{\mbox{ord}\, df^{k}(\rho(s))\mid 1\leq k\leq k_{L'}\}, 
\end{equation*}
and choose inductively  $1\leq k_j\leq k_{L'}$, $k_j\not=k_1,\ldots,k_{j-1}$, together with polynomials $c_{k_j,1}(s),\ldots,c_{k_j,j-1}(s)$
such that the order $\nu_{k_j}$ of
\begin{equation}\label{proof-notpkj}
\varphi^{k_j}(s) := df^{k_j}(\rho(s))+\sum_{l=1}^{j-1}c_{k_j,l}(s)\, df^{k_l}(\rho(s))
\end{equation}
is equal to the relative order\footnote{The relative order, $\mbox{ord}(df^{k_j}(\rho(s));\{k_1,\ldots,k_{j-1}\})$, of $df^{k_j}(\rho(s))$ modulo $\{k_1,\ldots,k_{j-1}\}$ is defined by
\begin{equation*}
\mbox{max} \biggl\{\mbox{ord} \biggl(df^{k_j}(\rho(s))+\sum_{l=1}^{j-1}c_l(s)\, df^{k_l}(\rho(s))\biggr)\mid c_1(s),\ldots,c_{j-1}(s)\in\mathbb{C}[s] \biggr\}.
\end{equation*}} 
$\mbox{ord}(df^{k_j}(\rho(s));\{k_1,\ldots,k_{j-1}\})$ and 
\begin{align*}
& \mbox{ord}(df^{k_j}(\rho(s));\{k_1,\ldots,k_{j-1}\})=\\
& \mbox{min} \{\mbox{ord}(df^{k}(\rho(s));\{k_1,\ldots,k_{j-1}\})
\mid k\not=k_1,\ldots,k_{j-1}\}.
\end{align*} 
Then, as $\varphi^{k_j}(s)=\omega_{k_j} s^{\nu_{k_j}}+\cdots$, we have
$\nu_{k_1}\leq\cdots\leq \nu_{k_{k_{L'}}}\leq d-w_{\mbox{\tiny min}}$
and $\omega_{k_1},\ldots,\omega_{k_{k_{L'}}}$ are linearly independent cotangent vectors in $T^*_{(\tau,\mathbf{q})}(\mathbb{C}\times \mathbb{C}^n)$. 
The last assertion can be proved by induction as follows. Let us assume that $\omega_{k_1},\ldots,\omega_{k_{j-1}}$ are linearly independent while $\omega_{k_1},\ldots,\omega_{k_j}$ are not. Then there are constants $c_{k_1},\ldots, c_{k_{j-1}}$ such that $\omega_{k_j}=\sum_{l=1}^{j-1} c_{k_l}\omega_{k_l}$. It follows that the order of
\begin{align*}
\psi^{k_j}(s) & :=\varphi^{k_j}(s)-\sum_{l=1}^{j-1} c_{k_l} s^{\nu_{k_j}-\nu_{k_l}}\varphi^{k_l}(s)\\
& = \bigg(\bigg(\sum_{l=1}^{j-1} c_{k_l}\omega_{k_l}\bigg)s^{\nu_{k_j}} + \cdots\bigg) - 
\sum_{l=1}^{j-1}c_{k_l} s^{\nu_{k_j}-\nu_{k_l}}(\omega_{k_l} s^{\nu_{k_l}}+\cdots)
\end{align*}
is greater than $\nu_{k_j}$ and $\psi^{k_j}(s)$ is the sum of $df^{k_j}(\rho(s))$ with a linear combination (whose coefficients are polynomials in $s$) of $df^{k_1}(\rho(s)),\ldots,df^{k_{j-1}}(\rho(s))$.
It follows that
\begin{align*}
\nu_{k_j}<\mbox{ord}\, \psi^{k_j}(s)\leq \mbox{ord}(df^{k_j}(\rho(s));\{k_1,\ldots,k_{j-1}\})=\nu_{k_j}, 
\end{align*}
which is a contradiction.
This completes the proof of Lemma \ref{lemma-fundlil}.
\end{proof}

\subsection{Construction of the forms $\omega_k$ corresponding to the functions $f^k$ such that $I\notin\mathcal{V}_{f^k_t}$}\label{proof-para2}
Let us now consider the set $L\setminus L'=\{k_{L'}+1,\ldots,k_L\}$ of indices $k$ such that $I\notin\mathcal{V}_{f^k_t}$. Then we have the following lemma.

\begin{lemma}\label{lemma-fundlil2}
For any $k_{L'}+1\leq k\leq k_{L}$, let us write
\begin{equation*}
df^k(\rho(s)) = \omega_k s^{\nu_k} + \cdots,
\end{equation*}
where $\omega_{k}\not=0$. Then $\nu_k=0$ and $\omega_{k_{L'}+1},\ldots,\omega_{k_{L}}$ are linearly independent cotangent vectors in $T^*_{(\tau,\mathbf{q})}(\mathbb{C}\times\mathbb{C}^n)$.
\end{lemma}

\begin{proof}
For simplicity, we shall write $f_\tau^{k,I}:=f_\tau^k\vert_{\mathbb{C}^I}$.
Since $I\notin\mathcal{V}_{f_\tau^k}$ for any $k_{L'}+1\leq k\leq k_{L}$, we have $d(\mathbf{w};f^k_\tau)=0$ and $f^k_{\tau,\mathbf{w}}=f_\tau^{k,I}$. Moreover, since $\rho(s)\in \overline{S^J(L)}$, the expression
\begin{equation*}
f^k(\rho(s))=f^k_{\tau,\mathbf{w}}(\mathbf{q})\, s^0+\cdots
\end{equation*}
vanishes for any $s$ and any $k_{L'}+1\leq k\leq k_{L}$, so that $\mathbf{q}\in V(f^k_{\tau,\mathbf{w}})$ for all $k_{L'}+1\leq k\leq k_{L}$.
The non-degeneracy condition says that by shrinking $(\tau,\mathbf{q})$ one more time if necessary, the variety 
\begin{equation*}
V(f^{k_{L'}+1}_\tau,\ldots,f^{k_{L}}_\tau)
\end{equation*} 
is a non-degenerate complete intersection variety, and hence, 
\begin{equation*}
df^{k_{L'}+1}_{\tau,\mathbf{w}}\wedge\cdots \wedge df^{k_{L}}_{\tau,\mathbf{w}} 
\end{equation*}
is nowhere vanishing in $V^*(f^{k_{L'}+1}_{\tau,\mathbf{w}},\ldots,f^{k_{L}}_{\tau,\mathbf{w}})$.
Since  for $k_{L'}+1\leq k\leq k_{L}$, the functions $f^k_{\tau,\mathbf{w}}$ involves only the variables $z_1,\ldots,z_{n_I}$, it follows that $k_L-k_{L'}\leq n_I$ and there are indices $i_{0,k_{L'}+1},\ldots,i_{0,k_{L}}\in I$ such that the determinant of the $((k_L-k_{L'})\times(k_L-k_{L'}))$-matrix
\begin{equation*}
\bigg(\frac{\partial f^{k}_{\tau,\mathbf{w}}}{\partial z_{i_{0,k'}}}(\mathbf{q})\bigg)_{k_{L'}+1\leq k,k'\leq k_L}
\end{equation*}
does not vanish. This implies that $\nu_k=0$. It also implies that $\omega_{k_{L'}+1}\wedge\cdots\wedge\omega_{k_{L}}$ has a non-zero coefficient in $dz_{i_{0,k_{L'}+1}}\wedge\cdots\wedge dz_{i_{0,k_{L}}}$, and hence $\omega_{k_{L'}+1},\ldots,\omega_{k_{L}}$ are linearly independent.
\end{proof}

\subsection{Consequence of Lemmas \ref{lemma-fundlil} and \ref{lemma-fundlil2}}\label{synthesis}
It follows from these lemmas that 
$\omega_{1}\wedge\cdots\wedge\omega_{k_{L'}}\wedge\omega_{k_{L'}+1}\wedge\cdots\wedge\omega_{k_{L}}\not=0$.
Moreover, since
\begin{equation*}
\bigwedge_{k=1}^{k_L}df^k(\rho(s)) = \bigwedge_{k=1}^{k_{L'}}\varphi^k(s) \wedge \bigwedge_{k=k_{L'}+1}^{k_L}df^k(\rho(s)),
\end{equation*}
where $\varphi^k(s)$ is defined in \eqref{proof-notpkj}, we deduce that $o_{k_L}=\nu_1+\cdots +\nu_{k_L}$ and $\omega_\infty$ is decomposable equal to $\omega_1\wedge \cdots \wedge \omega_{k_L}$. Thus, as announced above, to prove that $\ell_\infty\in\omega_\infty^\bot$, it suffices to show that $\omega_k(\ell_\infty)=0$ for all $1\leq k\leq k_L$. Indeed, $\imath_{\ell_\infty}(\omega_\infty)=\imath_{\ell_\infty}(\omega_1\wedge \cdots \wedge \omega_{k_L})=0$ if and only if
\begin{equation*}
\sum_{k=1}^{k_L} (-1)^{k-1}\omega_k(\ell_\infty) \, 
\omega_1\wedge \cdots \wedge \omega_{k-1} \wedge 
\omega_{k+1}\wedge \cdots \wedge \omega_{k_L} = 0.
\end{equation*}
So, if $\omega_k(\ell_\infty)=0$ for all $1\leq k\leq k_L$, then $\imath_{\ell_\infty}(\omega_\infty)=0$, that is, $\ell_\infty\in\omega_\infty^\bot$.

\subsection{Proof of the equality $\omega_k(\ell_\infty)=0$ for $1\leq k\leq k_{L'}$}\label{76}
By definition, if we put 
\begin{equation*}
\omega_k:=\omega_{k,0}\, dt + \sum_{i=1}^n\omega_{k,i}\, dz_i, 
\end{equation*}
then \eqref{egalitedulemme} is written as
\begin{equation*}
df^k(\rho(s))+\sum_{l=1}^{k-1}c_{k,l}(s)\, df^{l}(\rho(s)) =
\bigg(\omega_{k,0}\, dt + \sum_{i=1}^n\omega_{k,i}\, dz_i\bigg) s^{\nu_k}+\cdots.
\end{equation*}
Applying both sides of this equality to the canonical basis of $\mathbb{C}\times\mathbb{C}^n$ shows that
\begin{equation}\label{lir}
\left\{
\begin{aligned}
& \bigg(\frac{\partial {f^k_{\hat{\mathbf{w}}}}}{\partial t}(t_0,\mathbf{a}) + \sum_{l=1}^{k-1}c_{k,l}(0)\, \frac{\partial {f^l_{\hat{\mathbf{w}}}}}{\partial t}(t_0,\mathbf{a})\bigg)s^{d} +\cdots=\omega_{k,0}s^{\nu_k}+\cdots,\\
& \bigg(\frac{\partial f^k_{\tau,\mathbf{w}}}{\partial z_i}(\mathbf{a}) + \sum_{l=1}^{k-1}c_{k,l}(0)\, \frac{\partial f^l_{\tau,\mathbf{w}}}{\partial z_i}(\mathbf{a})\bigg)s^{d-w_i} +\cdots=\omega_{k,i}s^{\nu_k}+\cdots.
\end{aligned}
\right.
\end{equation}
As $\nu_k\leq d-w_{\mbox{\tiny min}}\leq d-1$ (see Lemma \ref{lemma-fundlil}), we must have $\omega_{k,0}=0$. We must also have $\omega_{k,i}=0$ for $1\leq i\leq n_I$ as $w_i=0$ within this range of indices.

\begin{remark}\label{remcrucial}
This implies $\lim_{s\to 0}T_{\rho(s)}V^{*J}(f^k)\supseteq \mathbb{C}\times \mathbb{C}^I$.
\end{remark}

Now, since $o_\ell\leq w_{\mbox{\tiny min}}$, the vector $\ell_\infty:=(\ell_{\infty,0},\ell_{\infty,1},\ldots,\ell_{\infty,n})$ is of the form
\begin{equation}\label{liml1}
\left\{
\begin{aligned}
&(\ell_{\infty,0},\ell_{\infty,1},\ldots,\ell_{\infty,n_I},0,\ldots,0) \mbox{ if } o_\ell<w_{\mbox{\tiny min}};\\
&(\ell_{\infty,0},\ell_{\infty,1},\ldots,\ell_{\infty,n_I},a_{n_I+1},\ldots,a_{n_I+m_I}, 0,\ldots,0)\mbox{ if } o_\ell=w_{\mbox{\tiny min}}.
\end{aligned}
\right.
\end{equation}
It follows that if $o_\ell<w_{\mbox{\tiny min}}$, then 
\begin{equation*}
\omega_k(\ell_\infty)=\underbrace{\omega_{k,0}}_{0}\ell_{\infty,0} +\sum_{i=1}^{n_I}\underbrace{\omega_{k,i}}_{0}\ell_{\infty,i} + \sum_{i=n_I+1}^n \omega_{k,i}\underbrace{\ell_{\infty,i}}_{0}=0.
\end{equation*}
If $o_\ell=w_{\mbox{\tiny min}}$, then the proof divides into two cases:
\begin{enumerate}
\item
$\nu_k< d-w_{\mbox{\tiny min}}$;
\item
$\nu_k = d-w_{\mbox{\tiny min}}$.
\end{enumerate}
In the first case, by \eqref{lir}, we must have $\omega_{k,i}=0$ not only for $0\leq i\leq n_I$ but also for $n_I+1\leq i\leq n_I+m_I$, and hence
\begin{equation*}
\omega_k(\ell_\infty)=\underbrace{\omega_{k,0}}_{0}\ell_{\infty,0} +
\sum_{i=1}^{n_I+m_I}\underbrace{\omega_{k,i}}_{0}\ell_{\infty,i} + 
\sum_{i=n_I+m_I+1}^n \omega_{k,i}\underbrace{\ell_{\infty,i}}_{0}=0,
\end{equation*}
as desired. In the second case, we still have $\omega_{k,i}=0$ for $0\leq i\leq n_I$ but we do not know whether $\omega_{k,i}=0$ for $n_I+1\leq i\leq n_I+m_I$. Thus, since $\ell_{\infty,i}=0$ for $n_I+m_I+1\leq i\leq n$, we have
\begin{equation*}
\omega_k(\ell_\infty)=\sum_{i=n_I+1}^{n_I+m_I}\omega_{k,i} \, \ell_{\infty,i}
=\sum_{i=n_I+1}^{n_I+m_I}\omega_{k,i} \, a_i.
\end{equation*}
Since $\rho(s)\in \overline{S^J(L)}$, we have $f^l\circ\rho(s)=0$, and hence $\frac{d(f^l\circ\rho)}{ds}(s)=0$, for any $s$ and any $l\in L$. It follows that 
\begin{align*}
\bigg(df^k(\rho(s))+\sum_{l=1}^{k-1}c_{k,l}(s)\, df^{l}(\rho(s))\bigg)
\bigg( \frac{d\rho}{ds}(s) \bigg)=
(\omega_k\, s^{\nu_k}+\cdots) \bigg( \frac{d\rho}{ds}(s) \bigg)=0
\end{align*}
for any $s$.
Then, by taking the coefficient of the term with the lowest degree (i.e., the coefficient of $s^{d-1}$), we obtain
\begin{equation}\label{egaliteazero}
\sum_{i=n_I+1}^n \bigg( \frac{\partial f^k_{\tau,\mathbf{w}}}{\partial z_i}(\mathbf{a}) + \sum_{l=1}^{k-1}c_{k,l}(0)\, \frac{\partial f^l_{\tau,\mathbf{w}}}{\partial z_i}(\mathbf{a}) \bigg) a_i w_i=0.
\end{equation}
But, by \eqref{lir} again, we have:
\begin{equation*}
\left\{
\begin{aligned}
& \omega_{k,i}=\frac{\partial {f^{k}_{\tau,\mathbf{w}}}}{\partial z_i} (\mathbf{a})+\sum_{l=1}^{k-1} c_{k,l}(0)\, \frac{\partial {f^l_{\tau,\mathbf{w}}}}{\partial z_i} (\mathbf{a}) \mbox{ for } n_I+1\leq i\leq n_I+m_I;\\
& \frac{\partial {f^{k}_{\tau,\mathbf{w}}}}{\partial z_i} (\mathbf{a})+
\sum_{l=1}^{k-1} c_{k,l}(0)\, \frac{\partial {f^l_{\tau,\mathbf{w}}}}{\partial z_i} (\mathbf{a})=0 \mbox{ for } n_I+m_I+1\leq i\leq n.
\end{aligned}
\right.
\end{equation*}
Thus \eqref{egaliteazero} is simply written as
\begin{align*}
0 = \sum_{i=n_I+1}^{n_I+m_I}\omega_{k,i}a_iw_i
= w_{\mbox{\tiny min}}\sum_{i=n_I+1}^{n_I+m_I} \omega_{k,i}\, a_i
= w_{\mbox{\tiny min}}\, \omega_k(\ell_\infty).
\end{align*}
Since $w_{\mbox{\tiny min}}>0$, again we have $\omega_k(\ell_\infty)=0$.

\subsection{Proof of the equality $\omega_k(\ell_\infty)=0$ for $k_{L'}+1\leq k\leq k_{L}$}\label{77}
Since $\rho(s)\in\overline{S^J(L)}$ and $\rho'(s)\in S^I(K)$, we have $f^k(\rho(s))=f^k(\rho'(s))=0$ for all $s$. Then, by the Taylor formula applied to the function $f^k$ at the point $\rho(s)$, there is a constant $c\in (0,1)$ such that for any $s\not=0$ small enough,
\begin{equation*}
0=\frac{f^k(\rho'(s))-f^k(\rho(s))}{\vert s\vert^{o_\ell}} = 
\sum_{i=0}^n \frac{\partial f^k}{\partial z_i}(\rho(s)) \, 
\frac{z'_i(s)-z_i(s)}{\vert s\vert^{o_\ell}} + \frac{1}{\vert s\vert^{o_\ell}}\, R(s),
\end{equation*}
where
\begin{equation*}
\frac{1}{\vert s\vert^{o_\ell}}\, R(s) \leq \frac{c}{\vert s\vert^{o_\ell}}\, \Vert \rho'(s)-\rho(s) \Vert^2 = \frac{c}{\vert s\vert^{o_\ell}}\, \Vert \ell(s) \Vert^2
\end{equation*}
tends to $0$ as $s\to 0$. (Here, to simplify, we have written $z_0(s)$ and $z_0'(s)$ instead of $t(s)$ and $t'(s)$.) It follows that
\begin{equation*}
df^k(\rho(s))\bigg( \frac{\ell(s)}{\vert s\vert^{o_\ell}} \bigg)= 
\sum_{i=0}^n \frac{\partial f^k}{\partial z_i}(\rho(s)) \, 
\bigg( \frac{\ell_i(s)}{\vert s\vert^{o_\ell}} \bigg)
\end{equation*}
tends to $0$ as $s\to 0$, that is, $\omega_k(\ell_\infty)=0$.

\subsection{Conclusion}\label{conclusion}
We have proved Theorem \ref{mt} if $I\not=\emptyset$. 
Actually, the argument developed above still works if $I=\emptyset$ and is even simpler. Indeed, in this case, \S\S \ref{proof-para2} and \ref{77} are no longer relevant, and since $\mathbf{w}$ is a positive weight vector, Lemma \ref{lemma-fundlil} just follows from the non-degeneracy of $V(f_t^1,\ldots,f_t^{k_L})$ and the constancy of the Newton diagram; the uniform local tameness is no longer needed.

\begin{remark}
A straightforward modification of the proof shows that Theorem \ref{mt} extends to multi-parameter deformation families. 
\end{remark}

\section{Direct proof of Theorem \ref{mt2}}\label{sect-proofmt2}

We have seen in Section \ref{sect-tafc} that Theorem \ref{mt2} can be obtained by combining Theorem \ref{mt}\textemdash which asserts that $\mathcal{S}$ is Whitney $(b)$-regular\textemdash with a theorem of Parusi\'nski \cite{P} and Brian\c con--Maisonobe--Merle \cite{BMM}\textemdash which says that Whitney $(b)$-regular stratifications always satisfy so-called $w_f$ condition, and hence, Thom's $a_f$ condition. 
Theorem \ref{mt2} can also be proved by a direct explicit calculation using only elementary methods. Hereafter we outline this argument.

By an appropriate version of \cite[Proposition (2.2)]{O}, to show that $\mathcal{S}$ satisfies Thom's $a_f$ condition, it suffices to prove that for any $I\subseteq\{1,\ldots,n\}$, any $K\subseteq K_0$, any $(\tau,\mathbf{q})=(\tau,q_1,\ldots,q_n)\in S^I(K)$ close enough to the origin, and any real analytic path $\rho(s)=(t(s),\mathbf{z}(s))$ with $\rho(0)=(\tau,\mathbf{q})$ and $\rho(s)\notin V(f)$ for $s\not=0$, the following condition on the tangent spaces holds true:
\begin{equation}\label{th2-rmet}
\lim_{s\to 0} T_{\rho(s)} V(f-f(\rho(s)))\supseteq T_{(\tau,\mathbf{q})}S^I(K).
\end{equation}

To simplify, as in the proof of Theorem \ref{mt}, we shall assume $\emptyset\not=I=\{1,\ldots,n_I\}$. Then we divide the argument into two cases. First, we consider the case where $K=K_0$, $I\notin\mathcal{V}_{f^k_t}$ for $1\leq k\leq k_0'\leq k_0$ and $I\in\mathcal{V}_{f^k_t}$ for $k_0'+1\leq k\leq k_0$. Then, without loss of generality, we may suppose that 
\begin{equation*}
\left\{
\begin{aligned}
& f^{k}_{\tau,\mathbf{w}}(\mathbf{a})=0  \mbox{ for }  k'_0+1\leq k\leq k_0''\leq k_0,\\
& f^{k}_{\tau,\mathbf{w}}(\mathbf{a})\not=0  \mbox{ for }  k''_0+1\leq k\leq k_0,
\end{aligned} 
\right.
\end{equation*} 
where, as above, $\mathbf{a}:=(a_1,\ldots,a_n)$ and $\mathbf{w}:=(w_1,\ldots,w_n)$ are defined by the Taylor expansions $z_i(s)=a_i s^{w_i}+\cdots$ ($1\leq i\leq n$). Put
$h:=f^{k_0''+1}\cdots f^{k_0}$ and decompose $f$ as
\begin{equation*}
f=(f^1\cdots f^{k'_0})\cdot (f^{k_0'+1}\cdots f^{k''_0})\cdot h.
\end{equation*}
By the proof of Proposition 3.2 in \cite{O33}, for any $s\not=0$, the tangent space $T_{\rho(s)} V(f-f(\rho(s)))$ contains the intersection
\begin{equation*}
\bigg(\bigcap_{k=1}^{k'_0} T_{\rho(s)} V(f^k-f^k(\rho(s)))\bigg) 
\cap T_{\rho(s)} V(f^{k_0'+1}\cdots f^{k''_0}\cdot h-f^{k_0'+1}\cdots f^{k''_0}\cdot h(\rho(s))).
\end{equation*}
By the non-degeneracy condition, $V^{*I}(f^k)$ is non-singular for any $1\leq k\leq k_0'$ (see \cite[Chap.~V, Lemma (2.8.2)]{O6}), and since $V^{*\{1,\ldots,n\}}(f^k)=V(f^k)\cap (\mathbb{C}\times \mathbb{C}^{*n})\ni (\tau,\mathbf{q})$ is a non-singular open subset of $V(f^k)$, the point $(\tau,\mathbf{q})$ is a non-singular point of $V(f^k)$. Thus, for any $1\leq k\leq k_0'$, we have
\begin{equation}\label{firstpart}
\lim_{s\to 0} T_{\rho(s)} V(f^k-f^k(\rho(s)))=T_{(\tau,\mathbf{q})} V(f^k) \supseteq T_{(\tau,\mathbf{q})} V^{*I}(f^k).
\end{equation}

The uniform local tameness of the family $\{V(f_t^{k_0'+1},\ldots,f_t^{k''_0})\}_t$ says that if $(\tau,\mathbf{q})$ is close enough to the origin, then 
\begin{equation*}
W:=V(f^{k'_0+1}_{\tau,\mathbf{w}},\ldots, f^{k''_0}_{\tau,\mathbf{w}})
\cap \mathbb{C}^{*n}(q_1,\ldots,q_{n_I})
\end{equation*} 
is a reduced, non-singular, complete intersection variety in $\mathbb{C}^{*n}(q_1,\ldots,q_{n_I})$. In particular, since $\mathbf{a}$ lies in $W$,
\begin{equation*}
df_{\tau,\mathbf{w}}^{k_0'+1}(\mathbf{a})\wedge \cdots \wedge 
df_{\tau,\mathbf{w}}^{k''_0}(\mathbf{a})\not=0.
\end{equation*}  
Now, since $h_{\tau,\mathbf{w}}=f_{\tau,\mathbf{w}}^{k_0''+1}\cdots f_{\tau,\mathbf{w}}^{k_0}$ is weighted homogeneous and $W$ is invariant under the associated $\mathbb{C}^*$-action on $\mathbb{C}^{*n}(q_1,\ldots,q_{n_I})$, the restriction of $h_{\tau,\mathbf{w}}$ to $W$ may only have zero as critical value. Since $h_{\tau,\mathbf{w}}(\mathbf{a})\not=0$, it follows that $dh_{\tau,\mathbf{w}}(\mathbf{a})\not\equiv 0$, and hence 
\begin{equation}\label{rajout}
dh_{\tau,\mathbf{w}}(\mathbf{a}),df_{\tau,\mathbf{w}}^{k_0'+1}(\mathbf{a}),\ldots, df_{\tau,\mathbf{w}}^{k''_0}(\mathbf{a})
\end{equation}
are linearly independent. Then, by an argument similar to that given in \S \ref{proof-para1}, the normalized limit of
\begin{equation*}
df^{k_0'+1}(\rho(s)) \wedge \cdots \wedge df^{k''_0}(\rho(s)) \wedge dh(\rho(s))
\end{equation*}
as $s\to 0$ is a decomposable form in $\mathbb{C}^{*n}(q_1,\ldots,q_{n_I})$, and
\begin{equation}\label{thirdpart}
\lim_{s\to 0} T_{\rho(s)}V(f^{k_0'+1}\cdots f^{k_0''}\cdot h-f^{k_0'+1}\cdots f^{k_0''}\cdot h(\rho(s))) \supseteq \mathbb{C}\times \mathbb{C}^I.
\end{equation}

Combining \eqref{firstpart} and \eqref{thirdpart} show that
\begin{align*}
\lim_{s\to 0}T_{\rho(s)} V(f-f(\rho(s))) 
\supseteq \bigg(\bigcap_{1\leq k\leq k_0'} T_{(\tau,\mathbf{q})} V^{*I}(f^k)\bigg) 
\cap (\mathbb{C}\times \mathbb{C}^{I}) 
\supseteq T_{(\tau,\mathbf{q})} S^I(K).
\end{align*}

Now, we look at the case where $K\varsubsetneq K_0$. To simplify, we assume $K=\{1,\ldots,k_K\}$, $I\notin\mathcal{V}_{f^k_t}$ for $1\leq k\leq k'_K\leq k_K$ (this range is empty if $I\in\mathcal{V}_{f^k_t}$ for all $k\in K$) and for $k_K+1\leq k\leq k_0$, and $I\in\mathcal{V}_{f^k_t}$ for all the other values of $k$. (Note that for $k_K+1\leq k\leq k_0$, the relation $I\notin\mathcal{V}_{f^k_t}$ is always true by definition of $S^I(K)$.) Then we have the following properties.

\begin{enumerate}
\item
For $1\leq k\leq k'_K$ (again this range may be empty), the non-degeneracy condition shows that $V^{*I}(f^k)$ is non-singular and $(\tau,\mathbf{q})$ is a non-singular point of $V(f^k)$. Thus, for all such $k$'s, we have:
\begin{equation*}
\lim_{s\to 0} T_{\rho(s)} V(f^k-f^k(\rho(s)))=T_{(\tau,\mathbf{q})} V(f^k) \supseteq T_{(\tau,\mathbf{q})} V^{*I}(f^k).
\end{equation*}
\item
Assume that $k'_K\geq 1$.
Since $(\tau,\mathbf{q})\in S^I(K)$, we have $f^{k_K'}(\tau,\mathbf{q})=0$ while $f^{k}(\tau,\mathbf{q})\not=0$ for $k_K+1\leq k\leq k_0$. By the non-degeneracy condition, $V^{*I}(f^{k_K'})$ is non-singular and $(\tau,\mathbf{q})$ is a non-singular point of $V(f^{k'_K})$, and hence, a non-singular point of $V(g)$, where $g:=f^{k'_K} f^{k_K+1}\cdots f^{k_0}$. Thus,
\begin{equation*}
\lim_{s\to 0} T_{\rho(s)} V(g-g(\rho(s))) = T_{(\tau,\mathbf{q})} V(g) = T_{(\tau,\mathbf{q})} V(f^{k_K'}) \supseteq T_{(\tau,\mathbf{q})} V^{*I}(f^{k_K'}).
\end{equation*}
\item
For $k'_K+1\leq k\leq k_K$, let us assume (to simplify) that 
\begin{equation*}
\left\{
\begin{aligned}
& f^{k}_{\tau,\mathbf{w}}(\mathbf{a})=0  \mbox{ for }  k'_K+1\leq k\leq k_K''\leq k_K,\\
& f^{k}_{\tau,\mathbf{w}}(\mathbf{a})\not=0  \mbox{ for }  k''_K+1\leq k\leq k_K.
\end{aligned} 
\right.
\end{equation*} 
Then, as in \eqref{rajout} and \eqref{thirdpart}, writing $u:=f^{k_K''+1}\cdots f^{k_K}$ and using the uniform local tameness condition, we show that 
\begin{equation*}
du_{\tau,\mathbf{w}}(\mathbf{a}), df^{k'_K+1}_{\tau,\mathbf{w}}(\mathbf{a}),\ldots,
df^{k''_K}_{\tau,\mathbf{w}}(\mathbf{a})
\end{equation*}
are linearly independent, and hence
\begin{equation*}
\lim_{s\to 0} T_{\rho(s)} V(f^{k_K'+1}\cdots f^{k''_K}\cdot u-f^{k_K'+1}\cdots f^{k''_K}\cdot u(\rho(s)))\supseteq \mathbb{C}\times \mathbb{C}^{I}.
\end{equation*}
\end{enumerate}

Combined with the proof of Proposition 3.2 in \cite{O33} again, properties (1)--(3) show that if $k'_K\geq 1$, then
\begin{align*}
\lim_{s\to 0}T_{\rho(s)} V(f-f(\rho(s))) 
\supseteq T_{(\tau,\mathbf{q})} S^I(K).
\end{align*}

To complete the proof, it remains to consider the case where $I\in\mathcal{V}_{f^k}$ for all $1\leq k\leq k_K$ (i.e., by abuse of language, the case where $k'_K=0$). By reordering, we may assume that 
\begin{align*}
\mbox{ord}\, f^{k_K}(\rho(s))=\mbox{min}\{\mbox{ord}\, f^k(\rho(s))\mid 1\leq k\leq k_K\}.
\end{align*}
Put $p:=f^{k_K}\cdot f^{k_K+1}\cdots f^{k_0}$. Since
\begin{align*}
p_{\tau,\mathbf{w}}=f^{k_K}_{\tau,\mathbf{w}}\cdot\prod_{k=k_K+1}^{k_0}({f^k_\tau}\vert_{\mathbb{C}^I})
\end{align*}
(see the proof of Lemma \ref{lemma-fundlil2}), it follows from the uniform local tameness condition that if $(\tau,\mathbf{q})$ is close enough to the origin, then
\begin{align*}
V(f^{1}_{\tau,\mathbf{w}},\ldots,f^{k_K-1}_{\tau,\mathbf{w}},p_{\tau,\mathbf{w}}) \cap
\mathbb{C}^{*n}(q_1,\ldots,q_{n_I})
\end{align*}
is a reduced, non-singular, complete intersection variety in $\mathbb{C}^{*n}(q_1,\ldots,q_{n_I})$. Then, by an argument similar to that given in \S \ref{proof-para1} again, we show that the normalized limit of
\begin{align*}
df^1(\rho(s))\wedge\cdots\wedge df^{k_K-1}(\rho(s))\wedge dp(\rho(s))
\end{align*}
 as $s\to 0$ is decomposable in $\mathbb{C}^{*n}(q_1,\ldots,q_{n_I})$, and 
\begin{align*}
\lim_{s\to 0} T_{\rho(s)} V(f-f(\rho(s))) \supseteq \mathbb{C}\times\mathbb{C}^I = T_{(\tau,\mathbf{q})} S^I(K).
\end{align*}

\section{Proof of Proposition \ref{proposition-us}}\label{proofofproposition-us}

We argue by contradiction. If the assertion fails, then (using the curve selection lemma) there is a real analytic path $\rho(s)=(t(s),\mathbf{z}(s))$ such that $\rho(0)=(0,\mathbf{0})$ and $df(\rho(s))\equiv 0$. We may assume that $\mathbf{z}(s)\in\mathbb{C}^{*n}$ for $s\not=0$. Consider the Taylor expansions 
\begin{equation*}
t(s)=t_0s^{w_0}+\cdots
\quad\mbox{and}\quad
z_i(s)=a_is^{w_i}+\cdots \ (1\leq i\leq n),
\end{equation*}
where $t_0, a_i\not=0$ and $w_0, w_i>0$, and set $\mathbf{a}:=(a_1,\ldots,a_n)$, $\mathbf{w}:=(w_1,\ldots,w_n)$ and $\hat{\mathbf{w}}:=(w_0,w_1,\ldots,w_n)$. Note that $d(\hat{\mathbf{w}};f^k)=d(\mathbf{w};f^k_0)$, and since $w_0>0$, $f^k_{\hat{\mathbf{w}}}(t,\mathbf{z})=f^k_{0,\mathbf{w}}(\mathbf{z})$ for all $1\leq k\leq k_0$. For simplicity, we assume that 
\begin{equation*}
\left\{
\begin{aligned}
& f^{k}_{0,\mathbf{w}}(\mathbf{a})=0  \mbox{ for }  1\leq k\leq k_0'\leq k_0,\\
& f^{k}_{0,\mathbf{w}}(\mathbf{a})\not=0  \mbox{ for }  k'_0+1\leq k\leq k_0.
\end{aligned} 
\right.
\end{equation*}
Let $d_k:=d(\mathbf{w};f^k_0)$. Then we have
\begin{equation*}
f^k (\rho(s)) = f^k_{0,\mathbf{w}}(\mathbf{a})\, s^{d_k}+\cdots 
\quad\mbox{and}\quad
\frac{\partial f^k}{\partial z_i} (\rho(s)) = \frac{\partial {f^k_{0,\mathbf{w}}}}{\partial z_i} (\mathbf{a})\,  s^{d_k-w_i}+\cdots.
\end{equation*}
Suppose that $k_0'\leq k_0-1$ and write $f=f^1\cdots f^{k_0'}\cdot h$, where $h:=f^{k_0'+1}\cdots f^{k_0}$. Then we have
\begin{align*}
\frac{\partial f}{\partial z_i}(\rho(s)) = \sum_{k=1}^{k_0'}
\bigg( \frac{\partial f^k}{\partial z_i}(\rho(s))\cdot h(\rho(s))\cdot\prod_{{1\leq k'\leq k_0'}\atop{k'\not=k}} f^{k'}(\rho(s))\bigg)+
\frac{\partial h}{\partial z_i}(\rho(s))\cdot\prod_{k=1}^{k_0'} f^{k}(\rho(s)).
\end{align*}
To simplify, put
\begin{equation*}
A_i^k(s):=\bigg( \frac{\partial f^k}{\partial z_i}(\rho(s))\cdot h(\rho(s))\cdot\prod_{{1\leq k'\leq k_0'}\atop{k'\not=k}} f^{k'}(\rho(s))\bigg)
\quad\mbox{and}\quad
B_i(s):=\frac{\partial h}{\partial z_i}(\rho(s))\cdot\prod_{k=1}^{k_0'} f^{k}(\rho(s)).
\end{equation*}
If $d'_k$ denotes the order (in $s$) of $f^k(\rho(s))$, then
\begin{equation}\label{pp42e1}
\mbox{ord}\, A_i^k(s)\geq
d_k-w_i+\sum_{k'=1}^{k_0'} d'_{k'}-d'_k+d(\mathbf{w};h_0).
\end{equation}
(Note that, as above, $d(\hat{\mathbf{w}};h)=d(\mathbf{w};h_0)$ and $h_{\hat{\mathbf{w}}}(t,\mathbf{z})=h_{0,\mathbf{w}}(\mathbf{z})$.)
Let $e_k:=d_k+\sum_{k'=1}^{k_0'} d'_{k'}-d'_k$, and for simplicity assume that
\begin{equation*}
e_{\mbox{\tiny min}}:=e_1=\cdots=e_{k_0''}<e_{k_0''+1}\leq \cdots\leq e_{k_0'}.
\end{equation*}
Then, \eqref{pp42e1} is written as $\mbox{ord}\, A_i^k(s)\geq
d(\mathbf{w};h_0)-w_i+e_k$, and since $d'_k>d_k$, we also have
\begin{equation}\label{pp42e2}
\mbox{ord}\, B_i(s) \geq d(\mathbf{w};h_0)-w_i+\sum_{k'=1}^{k_0'} d'_{k'} > 
d(\mathbf{w};h_0)-w_i+e_k.
\end{equation}
Now, since $\mathbf{a}\in V^*(f_{0,\mathbf{w}}^1,\ldots,f_{0,\mathbf{w}}^{k_0''})$ and $V(f_0^1,\ldots,f_0^{k_0''})$ is a non-degenerate complete intersection variety, we must have
\begin{equation}\label{proof-prop-iac}
df^1_{0,\mathbf{w}}(\mathbf{a})\wedge\cdots\wedge df^{k_0''}_{0,\mathbf{w}}(\mathbf{a})\not=0,
\end{equation}
and therefore there exist $i_1,\ldots,i_{k''_0}$ such that the determinant of the matrix
\begin{equation}\label{proof-smooth-matrix}
\bigg( \frac{\partial f^k_{0,\mathbf{w}}}{\partial z_{i_j}}(\mathbf{a}) \bigg)_{1\leq j,k\leq k''_0}
\end{equation}
is non-zero. In particular, for all $i_j\in \{i_1,\ldots,i_{k''_0}\}$ there is at least an index $k_{i_j}$ such that 
\begin{equation}\label{proof-smooth-dernn}
\frac{\partial {f^{k_{i_j}}_{0,\mathbf{w}}}}{\partial z_{i_j}} (\mathbf{a})\not=0.
\end{equation}
Let $m:=\mbox{min}\{d(\mathbf{w};h_0)-w_{i_j}+e_{\mbox{\tiny min}}\mid 1\leq j\leq k''_0\}$. Then we easily check that there exist non-zero complex numbers $c_1,\ldots,c_{k''_0}$ such that for any $i_j\in \{i_1,\ldots,i_{k''_0}\}$ we have
\begin{equation*}
\frac{\partial f}{\partial z_{i_j}}(\rho(s))=\sum_{k=1}^{k''_0}D_k\, s^m+\cdots,
\end{equation*}
where 
\begin{equation*}
D_k:=\left\{
\begin{aligned}
&\frac{\partial {f^k_{0,\mathbf{w}}}}{\partial z_{i_j}} (\mathbf{a})\cdot c_k && \mbox{if}
&& d(\mathbf{w};h_0)-w_{i_j}+e_{\mbox{\tiny min}} = m,\\
& 0 &&\mbox{if} && d(\mathbf{w};h_0)-w_{i_j}+e_{\mbox{\tiny min}} > m.
\end{aligned}
\right.
\end{equation*}
Using \eqref{proof-smooth-dernn}, it is easy to see that there is at least an index $i_{j_0}\in \{i_1,\ldots,i_{k''_0}\}$ and a non-empty subset $E\subseteq \{1,\ldots,k''_0\}$ such that
\begin{equation}\label{proof-smooth-rel-ld}
\frac{\partial f}{\partial z_{i_{j_0}}}(\rho(s))=\sum_{k\in E} \frac{\partial {f^k_{0,\mathbf{w}}}}{\partial z_{i_{j_0}}} (\mathbf{a})\cdot c_k\cdot s^m+\cdots,
\end{equation}
with $\frac{\partial {f^k_{0,\mathbf{w}}}}{\partial z_{i_{j_0}}} (\mathbf{a})\not=0$.
So, if $df(\rho(s))\equiv 0$, then the relation \eqref{proof-smooth-rel-ld} vanishes and gives a linear relation between the elements of the set
\begin{equation*}
\bigg\{\frac{\partial {f^k_{0,\mathbf{w}}}}{\partial z_{i_{j_0}}} (\mathbf{a})\bigg\} 
_{k\in E}.
\end{equation*} 
This implies that the determinant of the matrix \eqref{proof-smooth-dernn} vanishes\textemdash a contradiction to the Newton-admissibility condition.

The case $k'_0=k_0$ can be proved exactly in the same way. If $k'_0=0$ (i.e., if $f^k_{0,\mathbf{w}}(\mathbf{a})\not=0$ for $1\leq k\leq k_0$), then $f_{\hat{\mathbf{w}}}(t,\mathbf{z})=h_{\hat{\mathbf{w}}}(t,\mathbf{z})$ is a weighted homogeneous polynomial and the assertion is obvious in this case.

\begin{remark}
In the above proof, we have only used the non-degeneracy condition. The uniform local tameness assumption is not needed.
\end{remark}

\section{Proof of Proposition \ref{lemma-fpofmt3}}\label{proofprop-utr}

The argument is similar to that given in the proof of Corollary (2.8.1) in \cite{O6}. For completeness and convenience of the reader, we briefly recall and adapt it to our setting. We argue by contradiction. Choose a neighbourhood $\mathring{D}_\tau\times \mathring{B}_r$ of the origin of $\mathbb{C}\times\mathbb{C}^n$ such that Theorem \ref{mt} and Corollary \ref{cormt} hold, and suppose there exist a sequence $\{(t_m,\mathbf{z}_m)\}_m$ converging to $(t_0,\mathbf{0})\in \mathring{D}_\tau\times \mathring{B}_r$ and a stratum $S:=S^I(K)\in\mathcal{S}$ such that $\{t_m\}\times \mathbb{S}_{\Vert \mathbf{z}_m \Vert}$ is not transverse to $(\{t_m\}\times\mathbb{C}^n)\cap S$ at $(t_m,\mathbf{z}_m)$, that is,
\begin{equation}\label{proof-lemma-thmmt3-1}
T_{(t_m,\mathbf{z}_m)}((\{t_m\}\times\mathbb{C}^n)\cap S) \subseteq 
T_{(t_m,\mathbf{z}_m)}(\{t_m\}\times \mathbb{S}_{\Vert \mathbf{z}_m \Vert})\subseteq 
\{t_m\}\times \mathbf{z}_m^\bot,
\end{equation}
where $\mathbf{z}_m^\bot$ denotes the real orthogonal complement of $\mathbf{z}_m$.
By compactness, taking a subsequence if necessary, we may assume that the line defined by $\mathbf{z}_m$ and the origin in $\mathbb{C}^n$ converges to some $\ell$ in $\mathbb{P}^{n-1}$ and the following limits exist (in the appropriate Grassmannians):
\begin{align*}
& \cdot \ T:=\lim_{m\to\infty}T_{(t_m,\mathbf{z}_m)}((\{t_m\}\times\mathbb{C}^n)\cap S);\\
& \cdot \ T':=\lim_{m\to\infty}T_{(t_m,\mathbf{z}_m)} S.
\end{align*}
By \eqref{proof-lemma-thmmt3-1}, $T\subseteq \ell^\bot$. In particular, $\ell\not\subseteq T$. Now, by the Whitney $(a)$-regularity condition, we must have 
\begin{equation}\label{proof-lemma-thmmt3-2}
T'\supseteq T_{(t_0,\mathbf{0})}(\mathbb{C}\times \{\mathbf{0}\}),
\end{equation}
and since the line defined by $(t_m,\mathbf{z}_m)$ and $(t_m,\mathbf{0})$ in $\mathbb{C}\times\mathbb{C}^n$ converges to $\ell$ in $\mathbb{P}^n$, the Whitney $(b)$-regularity condition implies
\begin{equation}\label{proof-lemma-thmmt3-3}
\ell\subseteq T'.
\end{equation}
Combined with the relation $T'\supseteq T$, the inclusions \eqref{proof-lemma-thmmt3-2} and \eqref{proof-lemma-thmmt3-3} show that $T'$ contains the subspace $M$ generated by $\ell$, $T_{(t_0,\mathbf{0})}(\mathbb{C}\times \{\mathbf{0}\})$ and $T$. Therefore,
\begin{equation*}
n+1-k=\dim T'\geq \dim M = 1+1+(n-k)
\end{equation*}
($k$ is the number of defining equations of the stratum $S:=S^I(K)$), which is a contradiction.

\section{Proof of Theorem \ref{mt3}}\label{proofmt3}

Again, we argue by contradiction. Choose $\tau,r>0$ so that Theorem \ref{mt2}, Propositions \ref{proposition-us}, \ref{lemma-fpofmt3} and Corollary \ref{cor-mt2} hold, and suppose that the family $\{f_t\}$ does not have a uniform stable radius. Then, for all $0<\tau'\leq\tau$ and all $0<r'\leq r$, there exist $0<\varepsilon'\leq\varepsilon\leq r'$ such that for all sufficiently small $\delta>0$ there exist $\eta_\delta$, $\varepsilon_\delta$ and $t_\delta$, with $0<\vert\eta_\delta\vert\leq\delta$, $\varepsilon'\leq\varepsilon_\delta\leq\varepsilon$ and $\vert t_\delta\vert\leq\tau'$, such that $V(f_{t_\delta}-\eta_\delta)$ is non-singular in $\mathring{B}_{r'}$ and does not transversely intersect the sphere $\mathbb{S}_{\varepsilon_\delta}$ at some point $\mathbf{z}_\delta$, so that
\begin{equation}\label{pmth3-itss}
T_{\mathbf{z}_\delta} V(f_{t_\delta}-\eta_\delta) \subseteq 
T_{\mathbf{z}_\delta} \mathbb{S}_{\varepsilon_\delta} = 
T_{\mathbf{z}_\delta} \mathbb{S}_{\Vert {\mathbf{z}_\delta} \Vert}.
\end{equation}
Take $\delta:=\delta_m:=1/m$ and consider the corresponding sequence $\{(t_{\delta_m},\mathbf{z}_{\delta_m})\}_m$  in $D_{\tau'}\times (B_\varepsilon\setminus \mathring{B}_{\varepsilon'})$. By compactness, we may assume the following limits~exist:
\begin{align*}
& \cdot \ T:=\lim_{m\to \infty} T_{\mathbf{z}_{\delta_m}} 
V(f_{t_{\delta_m}}-\eta_{\delta_m});\\
& \cdot \ T':=\lim_{m\to \infty} T_{(t_{\delta_m},\mathbf{z}_{\delta_m})} 
V(f-\eta_{\delta_m});\\
& \cdot \ (t_0,\mathbf{z}_0):=\lim_{m\to \infty}(t_{\delta_m},\mathbf{z}_{\delta_m})
\in D_{\tau'}\times (B_\varepsilon\setminus \mathring{B}_{\varepsilon'}).
\end{align*}
Then, by \eqref{pmth3-itss}, we have:
\begin{equation}\label{pmth3-itss-2}
T\subseteq 
\lim_{m\to \infty} T_{\mathbf{z}_{\delta_m}} \mathbb{S}_{\Vert \mathbf{z}_{\delta_m} \Vert} = T_{\mathbf{z}_0} \mathbb{S}_{\Vert \mathbf{z}_0 \Vert}.
\end{equation}
Let $S_{00}$ be the stratum of $\mathcal{S}$ that contains $(t_0,\mathbf{z}_0)$, and let $S_0:=(\{t_0\}\times \mathbb{C}^n)\cap S_{00}$ be the stratum of $\mathcal{S}_{t_0}$ that contains $\mathbf{z}_0$ (in particular, $\{t_0\}\times \mathbb{C}^n$ and $S_{00}$ transversely intersect at $(t_0,\mathbf{z}_0)$). By Thom's $a_f$ condition, $T_{(t_0,\mathbf{z}_0)} S_{00}\subseteq T'$, and  hence,
\begin{equation}\label{pp52-its}
T_{(t_0,\mathbf{z}_0)} S_{0} \subseteq (\{t_0\}\times \mathbb{C}^n) \cap T'.
\end{equation}
This implies that $\{t_0\}\times \mathbb{C}^n$ and $T'$ transversely intersect at $(t_0,\mathbf{z}_0)$ (as otherwise $\{t_0\}\times \mathbb{C}^n$ and $S_{00}$ would not be transverse at $(t_0,\mathbf{z}_0)$ either). By taking the limit as $m\to\infty$ of the expression
\begin{align*}
\{t_{\delta_m}\}\times T_{\mathbf{z}_{\delta_m}}V(f_{t_{\delta_m}}-\eta_{\delta_m})&\equiv
T_{(t_{\delta_m},\mathbf{z}_{\delta_m})}(\{t_{\delta_m}\}\times V(f_{t_{\delta_m}}-\eta_{\delta_m}))\\
&=(\{t_{\delta_m}\}\times \mathbb{C}^n) \cap T_{(t_{\delta_m},\mathbf{z}_{\delta_m})}
V(f-\eta_{\delta_m}),
\end{align*}
the transversality of $\{t_0\}\times \mathbb{C}^n$ and $T'$ at $(t_0,\mathbf{z}_0)$ implies
\begin{equation*}
\{t_0\}\times T = (\{t_0\}\times \mathbb{C}^n) \cap T'.
\end{equation*}
Combined with \eqref{pmth3-itss-2} and \eqref{pp52-its}, this shows
\begin{equation*}
T_{\mathbf{z}_0} S_0\equiv T_{(t_0,\mathbf{z}_0)} S_0\overset{\eqref{pp52-its}}{\subseteq} (\{t_0\}\times \mathbb{C}^n) \cap T'=\{t_0\}\times T\overset{\eqref{pmth3-itss-2}}{\subseteq} \{t_0\}\times T_{\mathbf{z}_0} \mathbb{S}_{\Vert \mathbf{z}_0 \Vert} \equiv
T_{\mathbf{z}_0} \mathbb{S}_{\Vert \mathbf{z}_0 \Vert},
\end{equation*}
and therefore $S_0$ and $\mathbb{S}_{\Vert \mathbf{z}_0 \Vert}$ do not meet tranversely at $\mathbf{z}_0$, which contradicts Proposition \ref{lemma-fpofmt3}.

\bibliographystyle{amsplain}

\end{document}